\documentclass[12pt]{amsart}

\usepackage{amsmath}
\usepackage{amssymb}
\usepackage{mathdots}
\usepackage[all]{xy}

\numberwithin{equation}{section}

\newtheorem{thm}{Theorem}[section]

\newtheorem{lem}[thm]{Lemma}
\newtheorem{prop}[thm]{Proposition}
\newtheorem{cor}[thm]{Corollary}
\theoremstyle{definition}
\newtheorem{rem}[thm]{Remark}

\newtheorem{exam}[thm]{Example}
\newtheorem{exam-nota}[thm]{Example-Notation}
\newtheorem{nota}[thm]{Notation}
\newtheorem{dfn}[thm]{Definition}

\newtheorem{dfn-nota}[thm]{Definition-Notation}

\newtheorem{dfn-lem}[thm]{Lemma-Definition}

\newcommand{\beqa}{\begin{eqnarray*}}
\newcommand{\eeqa}{\end{eqnarray*}}

\newcommand{\id}{\mbox{${\rm id}$}}
\newcommand{\Id}{\mbox{${\rm Id}$}}

\newcommand{\Hom}{\mbox{${\rm Hom}$}}

\newcommand{\fc}{\mbox{${\mathfrak c}$}}

\newcommand{\fa}{\mbox{${\mathfrak a}$}}
\newcommand{\ft}{\mbox{${\mathfrak t}$}}
\newcommand{\fk}{\mbox{${\mathfrak k}$}}
\newcommand{\fg}{\mbox{${\mathfrak g}$}}

\newcommand{\fl}{\mbox{${\mathfrak l}$}}
\newcommand{\fs}{\mbox{${\mathfrak s}$}}
\newcommand{\fsl}{\mbox{${\fs\fl}$}}

\newcommand{\fh}{\mbox{${\mathfrak h}$}}
\newcommand{\fn}{\mbox{${\mathfrak n}$}}

\newcommand{\fp}{\mbox{${\mathfrak p}$}}
\newcommand{\fr}{\mbox{${\mathfrak r}$}}
\newcommand{\fb}{\mbox{${\mathfrak b}$}}
\newcommand{\fz}{\mbox{${\mathfrak z}$}}
\newcommand{\fm}{\mbox{${\mathfrak m}$}}

\newcommand{\eps}{\epsilon}

\newcommand{\C}{\mbox{${\mathbb C}$}}

\newcommand{\Ad}{{\rm Ad}}

\newcommand{\fgl}{\mathfrak{gl}}

\newcommand{\ad}{\operatorname{ad}}

\newcommand{\B}{\mathcal{B}}

\newcommand{\fso}{\mathfrak{so}}

\newcommand{\codim}{\mbox{codim}}
\newcommand{\N}{\mathfrak{N}}
\newcommand{\Kbundleslice}{K\times_{H\cap K} (x+\mathfrak{N}_{\fg^{-\theta}})}
 
\newcommand{\partialfibre}{\Phi_{n}^{-1}(\Phi_{n}(x))}

\title{The Complex Orthogonal Gelfand-Zeitlin system }

\author[M. Colarusso]{Mark Colarusso}
\address{Department of Mathematics and Statistics, University of South Alabama, Mobile, AL 36688}
\email{mcolarusso@southalabama.edu}

\author[S. Evens]{Sam Evens}
\address{Department of Mathematics, University of Notre Dame, Notre Dame, IN, 46556}
\email{sevens@nd.edu}

\begin{document}
\maketitle

\begin{abstract}
 In this paper, we use the theory of algebraic groups to prove a number of
new and fundamental results about 
the orthogonal Gelfand-Zeitlin system.  We show that the moment map (orthogonal Kostant-Wallach map) is surjective and simplify 
criteria of Kostant and Wallach for an element to be strongly regular.  We further prove the integrability of the orthogonal 
Gelfand-Zeitlin system on regular adjoint orbits and describe the generic flows of the integrable system.  We also study the nilfibre of the moment map and show that in contrast to the general linear case it contains no strongly regular elements.  This extends results of Kostant, Wallach, and Colarusso from the
general linear case to the orthogonal case.
\end{abstract}

\section{Introduction}\label{s:intro}

Kostant and Wallach studied the Gelfand-Zeitlin system of functions for
the Lie algebra of complex $n$ by $n$ matrices in two fundamental papers
 \cite{KW1, KW2}, and found a number of new features that did not appear
in the Gelfand-Zeitlin system of functions on Hermitian matrices \cite{GS}.
The Gelfand-Zeitlin functions generate a Poisson commutative family of
functions and in \cite{KW1}, Kostant and Wallach show they generate an
integral system on each regular conjugacy class of matrices.  To show this, they 
introduced and studied the set of {\it strongly regular} matrices,
which is the set where the Gelfand-Zeitlin functions have linearly 
independent differentials.   There is also a Gelfand-Zeitlin system on
the complex orthogonal Lie algebra $\fso(n,\C)$.   Many fundamental
and incisive linear algebra constructions used to study matrices in \cite{KW1,Col1}
 do not carry over to the orthogonal case, and as a consequence, much less is
known about the complex orthogonal Gelfand-Zeitlin system.  The purpose of this paper and its sequel is to remedy this situation.  We 
establish the complete integrability of the orthogonal Gelfand-Zeitlin system on regular adjoint orbits in $\fso(n,\C)$ and extend a number of basic results 
on the strongly regular set to $\fso(n,\C)$, which were established in the general linear case in \cite{KW1, Col1, CEKorbs, CEeigen}.  In particular, we describe the generic leaves of the foliation given by the integrable system as well as aspects of the geometry of the nilfibre of the moment 
map of the system.  The key idea is to use results from the theory of algebraic group actions due to Knop, Panyushev, and Luna to extend the general 
linear results to the orthogonal setting.  As a consequence, we understand
the general linear and orthogonal complex Gelfand-Zeitlin systems using a
unified approach, and avoid some of the subtle linear algebra calculations
from previous approaches.

 %The purpose
%of this paper and its sequel is to remedy this situation.   The key
%point is to use results from the theory of algebraic group actions due to
%Knop, Luna, and Panyushev in order to extend results from the case
%of matrices to the orthogonal case, but we also get better insight
%into the original general linear case.   We are able to prove the
%complete integrability of the orthogonal Gelfand-Zeitlin system on
%regular adjoint orbits in $\fso(n,\C)$, and extend a number of basic
%results on the strongly regular set to $\fso(n,\C)$, which were established
%in the general linear case in \cite{KW1, Col1, Col2, CEeigen}. MARK, DID I SAY
%THIS CORRECTLY? 

In more detail, we let $\fg = \fg_n = \fso(n,\C)$, and fix a sequence
of embeddings $\fg_2 \subset \fg_3 \subset \dots \subset \fg_n = \fg$,
and let $G_n = SO(n,\C)$.  Using an invariant form, we embed a choice of polynomial generators of $\C[\fg_{i}]^{G_{i}}$ into $\C[\fg]$, 
the regular functions on $\fg$, and we let $J_{GZ}$ be the collection of functions consisting of the polynomial generators
of $\C[\fg_{i}]^{G_{i}}$ for $i=2,\dots, n$.  Letting $r_i$ be the rank
of $\fg_i$, we obtain a Kostant-Wallach (or KW) morphism
$\Phi: \fg \to \C^{r_2} \times \dots\times \C^{r_n}$ with coordinates given by the
above polynomial generators.  We also use extensively the partial KW morphism $\Phi_n:\fg \to \C^{r_{n-1}} \times \C^{r_n}$ defined
using the polynomial generators from $\C[\fg_{n-1}]^{G_{n-1}}$ and
$\C[\fg_n]^{G_n}$.   We note that $\Phi_n$ may be identified with
a geometric invariant theory quotient $\fg \to \fg//G_{n-1}$, and
combine this with an interlacing argument and a basic flatness
result of Knop \cite{Kn} for spherical varieties to prove our first main
theorem.

\begin{thm}\label{thm:surjintro} (see Theorem \ref{thm:surj})
The morphism $\Phi$ is surjective, and every fibre contains a regular
element of $\fg$.
\end{thm}

An element $x \in \fg$ is called {\it strongly regular} if the set
$\{ df(x) : f \in J_{GZ} \}$ is linearly independent.   In the
case of $\fgl(n,\C)$, Kostant and Wallach consider the sequence of subalgebras 
$\fg_1\subset \fg_2 \subset \dots \subset \fg_n = \fgl(n,\C)$ by letting
$\fg_i \cong \fgl(i,\C)$ be the upper left $i$ by $i$ corner, and given $x\in \fgl(n,\C)$, define
$x_i \in \fg_i$ by orthogonal projection. They define
Gelfand-Zeitlin functions and the strongly regular set as above, and 
prove that $x\in \fgl(n,\C)$ is strongly regular if and only
if (i) $x_i$ is regular in $\fg_i$ for each $i$ and (ii) the centralizers
 $\fz_{\fg_i}(x_i) \cap \fz_{\fg_{i+1}}(x_{i+1}) = 0$ for $i=1, \dots, n-1$.
We use results on spherical varieties due to Panyushev  \cite{Pancoiso}
to prove that we may omit condition (i), and extend this result to the
case of $\fso(n,\C)$ (Proposition \ref{prop:fullsreg}).   Here, given $x \in \fg = \fso(n,\C)$, we also define $x_i \in \fg_i \cong \fso(i,\C)$ by orthogonal projection.  
Using this simplified criterion for strong regularity, we can construct a large collection of strongly regular elements.  In particular, we consider the
set $\fg_{\Theta}$ in $\fso(n,\C)$ given by the property that the 
(suitably defined) spectra of $x_i$ and $x_{i+1}$ do not intersect for
$i=2, \dots, n-1$ (see Notation \ref{nota:spectrum} and Equation \ref{eq:fgtheta}).   In Proposition \ref{prop:orthofgtheta}, we prove that elements
of $\fg_{\Theta}$ are strongly regular.   We then use this result, together with 
Theorem \ref{thm:surjintro}, to prove one of our main results.

\begin{thm}\label{thm:intsystemintro} (see Theorem \ref{thm:intsystem})
The restriction of the Gelfand-Zeitlin functions $J_{GZ}$ to a regular adjoint orbit in $\fg$ forms a completely integrable system on the orbit. 
\end{thm}

We further show that for $x\in \fg_{\Theta}$,
the fibre $\Phi^{-1}(\Phi(x))$ has a free action by an abelian
linear algebraic group, and thereby extend a result of Kostant and Wallach
and the first author \cite{KW1, Col1} from the general linear case to the
orthogonal case (Theorem \ref{thm:GZgroup}).  In the final section, we study the nilfibre $\Phi^{-1}(0)$.  
We begin by studying the partial nilfibre $\Phi_{n}^{-1}(0)$.  We show that its irreducible components can be 
described in terms of closed $G_{n-1}=SO(n-1,\C)$-orbits on the flag variety $\B$ of $\fg$ (Theorem \ref{thm:partialnil}).  
We accomplish this by using the Luna slice theorem to describe the generic fibres of $\Phi_{n}$ and then 
degenerate a generic fibre to $\Phi_{n}^{-1}(0)$ using Knop's flatness result.  Using another interlacing argument and well-known facts about the closed $G_{n-1}$-orbits on $\B$, we prove: 

\begin{prop}\label{p:intronil} (see Proposition \ref{p:nosreg})
The nilfibre of $\Phi$ contains no strongly regular elements.
\end{prop}
This stands in contrast to the 
case of $\fgl(n,\C)$ studied extensively in \cite{CEKorbs}.  As a consequence of Proposition \ref{p:intronil}, we prove that there is no analogue of the Hessenberg
matrices, which play a fundamental role in \cite{KW1} (Corollary \ref{c:noHess}).   In the sequel,
we plan to further develop these methods to completely understand the
strongly regular set, and describe in full detail each partial KW fiber
$\Phi_n^{-1}(\Phi_n(x))$.

This paper is organized as follows.  In Section \ref{s:realization},
we introduce notation and results for later use.  In 
Section \ref{s:surjectivity}, we identify the partial KW map
$\Phi_n$ as an invariant theory quotient and show that $\Phi$ is
surjective.   In Section \ref{s:spherical}, we study the strongly regular
set, and prove that $\fg_{\Theta}$ consists of strongly
regular elements.   We further prove complete integrability of regular
orbits, and study the KW fibers for elements in $\fg_{\Theta}$.  In Section \ref{s:KWnilfibre}, we study the nilfibres of $\Phi_{n}$ and $\Phi$ 
and prove that the nilfibre of $\Phi$ contains no strongly regular elements. 
In the body of the paper, all Lie algebras are complex, as are all algebraic
groups.   In particular, we will write $\fso(n)$ and $SO(n)$ to denote
 $\fso(n,\C)$ and $SO(n,\C)$, and similarly with $\fgl(n)$.

The first author was supported in part by
 NSA grant number H98230-16-1-0002 and the second author was supported in part by Simons Foundation Travel Grant 359424. 
 
We would like to thank Nolan Wallach and Jeb Willenbring for useful discussions relevant to the subject of this paper. 

\section{Preliminaries}\label{s:realization}

In this section, we recall basic facts about the orthogonal Lie algebras 
 (Section \ref{ss:orthoreal}).   After introducing some notation, we review basic properties of the orthogonal 
Gelfand-Zeitlin systems (hereafter referred to as GZ systems) and its moment map (orthogonal Kostant-Wallach map) and summarize known results (Section \ref{ss:GZsystems}).  We begin by 
describing the realization of $\fso(n)$ that we will use throughout the paper.  

\subsection{Realization of Orthogonal Lie algebras}\label{ss:orthoreal}  
  We give explicit descriptions of standard Cartan subalgebras
and corresponding root systems of $\fso(n)$.  Our exposition follows Chapters 1 and 2 of \cite{GW}.  

Let $\beta$ be the non-degenerate, symmetric 
bilinear form on $\C^{n}$ given by 
\begin{equation}\label{eq:beta}
\beta(x,y)=x^{T} S_{n} y, 
\end{equation}
where $x, y$ are $n\times 1$ column vectors and $S_{n}$ is the $n\times n$ 
matrix:
\begin{equation}\label{eq:Sn}
S_{n}=\left[\begin{array}{ccccc}
0& \dots &\dots & 0 & 1\\
\vdots &  & & 1 & 0\\
\vdots &  &\iddots  & & \vdots\\
0& 1& \dots & 0 & \vdots\\
1 & 0& \dots & \dots & 0\end{array}\right]
\end{equation}
with ones down the skew diagonal and zeroes elsewhere.  
The special orthogonal group,
$SO(n)$, consists of the $g\in SL(n)$ such that $\beta(gx, gy)=\beta(x,y)$ for all $x,\, y \in \C^{n}$.
Its Lie algebra,
$\fso(n)$, consists of the $Z\in\fsl(n)$ such that $\beta(Zx, y)=-\beta(x,Zy)$ for
all $x, \, y\in\C^{n}$.  
We consider the cases where $n$ is odd and even separately.  Throughout, we denote the standard basis of $\C^{n}$ by $\{e_{1},\dots, e_{n}\}$.

\subsubsection{Realization of $\fso(2l)$} \label{ss:soevenreal}
Let $\fg=\fso(2l)$ be of type $D$.  The subalgebra of diagonal matrices $\fh:=\{\mbox{diag}[a_{1},\dots, a_{l}, -a_{l},\dots, -a_{1}],\, a_{i}\in\C\}$ is a Cartan subalgebra of $\fg$.  We refer to $\fh$ as the \emph{standard Cartan subalgebra}.  
Let $\epsilon_{i}\in\fh^{*}$ be the linear functional $\epsilon_{i}(\mbox{diag}[a_{1},\dots, a_{l}, -a_{l},\dots, -a_{1}])=a_{i}$, and 
let $\Phi(\fg, \fh)$ be the roots of $\fg$ with respect to $\fh$.  Then 
$\Phi(\fg,\fh)=\{\epsilon_{i}-\epsilon_{j},\, \pm(\epsilon_{i}+\epsilon_{j}):\; 1\leq i\neq j\leq l\}$, and
we take as our \emph{standard positive roots} the set
$\Phi^{+}(\fg,\fh):=\{\epsilon_{i}-\epsilon_{j},\, \epsilon_{i}+\epsilon_{j}:\; 1\leq i< j\leq l\}$
with corresponding simple roots $\Pi:=\{\alpha_{1},\dots, \alpha_{l-1}, \alpha_{l}\}$
where  $\alpha_{i}=\epsilon_{i}-\epsilon_{i+1}$ for $i=1, \dots, l-1$, and $\alpha_{l}=\epsilon_{l-1}+\epsilon_{l}$.  The Borel subalgebra $\fb_{+}:=\fh\oplus\displaystyle\bigoplus_{\alpha\in\Phi^{+}(\fg,\fh)} \fg_{\alpha}$ is easily seen to be the set of upper triangular matrices in $\fg$. 

For the purposes of computations with $\fso(2l)$, it is convenient to 
relabel part of the standard basis of $\C^{2l}$ as $e_{-j}:=e_{2l+1-j}$ for $j=1, \dots, l$.  
%Then the basis $\{e_{\pm 1}, \dots, e_{\pm l}\}$ is an isotropic basis of $\C^{2l}$ with respect to the form $\beta$ in (\ref{eq:beta}).

\subsubsection{Realization of $\fso(2l+1)$} \label{ss:sooddreal}
Let $\fg=\fso(2l+1)$ be of type $B$.  The subalgebra of diagonal matrices $\fh:=\{\mbox{diag}[a_{1},\dots, a_{l}, 0, -a_{l},\dots, -a_{1}],\, a_{i}\in\C\}$ is a Cartan subalgebra of $\fg$.  We again refer to $\fh$ as the \emph{standard Cartan subalgebra}.  
Let $\epsilon_{i}\in\fh^{*}$ be the linear functional $\epsilon_{i}(\mbox{diag}[a_{1},\dots, a_{l},0,  -a_{l},\dots, -a_{1}])=a_{i}$.  In this case, the roots are
$\Phi(\fg,\fh)=\{\epsilon_{i}-\epsilon_{j},\, \pm(\epsilon_{i}+\epsilon_{j}):\; 1\leq i\neq j\leq l\} \cup \{\pm \epsilon_{k} :\, 1\leq k\leq l\}$.
We take as our \emph{standard positive roots} the set
$\Phi^{+}(\fg,\fh):=\{\epsilon_{i}-\epsilon_{j},\, \epsilon_{i}+\epsilon_{j}:\; 1\leq i< j\leq l\}\cup \{ \epsilon_{k} :\, 1\leq k\leq l\}$ 
with corresponding simple roots 
$\Pi:=\{\alpha_{1},\dots, \alpha_{l-1}, \alpha_{l}\}$ where $\alpha_{i}=\epsilon_{i}-\epsilon_{i+1}, \, i=1, \dots, l-1, \, \alpha_{l}=\epsilon_{l}$.
The Borel subalgebra $\fb_{+}:=\fh\oplus\displaystyle\bigoplus_{\alpha\in\Phi^{+}(\fg,\fh)} \fg_{\alpha}$
is easily seen to be the set of upper triangular matrices in $\fg$.
 
  We relabel part of the standard basis of $\C^{2l+1}$ by
letting $e_{-j}:=e_{2l+2-j}$ for $j=1, \dots, l$ and $e_{0}:=e_{l+1}$.  
%Then the basis $\{e_{\pm 1}, \dots, e_{\pm l}, e_{0}\}$ is an isotropic basis of $\C^{2l+1}$ with respect to the form $\beta$ in (\ref{eq:beta}).  

%We omit the description of the standard root system for $\fgl(n,\C)$, since we will not be performing 
%matrix computations with $\fgl(n,\C)$.  

\subsection{Split Rank $1$ symmetric subalgebras}\label{ss:symmetricreal}

For later use, 
recall the realization of $\fso(n-1)$ as a symmetric 
subalgebra of $\fso(n)$.
For $\fg=\fso(2l+1)$, let $t$ be an element of the Cartan subgroup with
Lie algebra $\fh$  with the property that 
$\Ad(t)|_{\fg_{\alpha_i}}=\id$ for $i=1, \dots, l-1$ and $\Ad(t)|_{\fg_{\alpha_l}}=-\id$.  Consider the involution $\theta_{2l+1}:=\Ad(t)$.   Then $\fk=\fso(2l)=\fg^{\theta_{2l+1}}$ 
(see \cite{Knapp02}, p. 700).  Note that $\fh \subset \fk$.
%It follows that $\fg^{\theta}=\{ A=(a_{ij}) \in \fso(2l+1, \C) : a_{lj}=a{kl}=0, \forall 1 \le j, k \le 2l+1 \}.$
In the case $\fg=\fso(2l)$, $\fk=\fso(2l-1)=\fg^{\theta_{2l}}$,
where $\theta_{2l}$ is the involution induced by the diagram automorphism interchanging
the simple roots $\alpha_{l-1}$ and $\alpha_l$ relative to a fixed choice
of simple root vectors (see \cite{Knapp02}, p. 703).   Note that
in this case, $\theta_{2l}(\eps_l)=-\eps_l$ and $\theta_{2l}(\eps_i)=\eps_i$ for $i=1, \dots, l-1$.  
We will omit the subscripts $2l+1$ and $2l$ from $\theta$
when $\fg$ is understood.  % Note that in either case, $\fh \cap \fk$ is
%a Cartan subalgebra of $\fk$.
%I DON'T THINK WE NEED THIS STATEMENT ABOUT THE CARTAN CONSIDERING NEW MATERIAL IN SECTION 5.

We also denote the corresponding involution of $G=SO(n)$ by $\theta$.  
The fixed subgroup $G^{\theta}=S(O(n-1)\times O(1))$ is disconnected.  
We let $K:=(G^{\theta})^{0}$ be the identity component of $G^{\theta}$.  Then $K=SO(n-1)$, and $\mbox{Lie}(K)=\fk=\fg^{\theta}$.

In both cases, $\theta$ preserves the standard Cartan subalgebra $\fh$,
and hence acts on the roots $\Phi(\fg,\fh)$.
A root $\alpha$ is called {\it real}
if $\theta(\alpha)=-\alpha$, {\it imaginary} if $\theta(\alpha)=\alpha$, 
and {\it complex } if $\theta(\alpha)\not= \pm \alpha.$   If $\alpha$ is
imaginary, then $\alpha$ is called {\it compact} if $\theta|_{\fg_\alpha}=\id$
and {\it noncompact} if $\theta|_{\fg_\alpha}=-\id.$    If $\alpha$ is complex, then $\alpha$ is called \emph{complex} $\theta$-{\it stable} if $\theta(\alpha)$
is positive, and otherwise is called \emph{complex} $\theta$-{\it unstable}.  

\begin{exam}\label{ex:roottypes}
Let $\fg=\fso(2l+1)$ and $\fk=\fso(2l)$, and let $\theta=\Ad(t)$ be 
as above.  The roots $\{\pm(\epsilon_{i}-\epsilon_{j}),\, \pm(\epsilon_{i}+\epsilon_{j}), \;1\leq i < j\leq l\}$ are compact imaginary, and the roots $\{ \pm \epsilon_{i}\; i=1,\dots, l\}$ are non-compact imaginary.  As noted above $\fh\subset\fk$, so that $\fh$ is a Cartan subalgebra of $\fk$, and it is easy to see that the set of roots $\{\alpha_{1},\,\alpha_{2},\,\dots,\, \alpha_{l-1},\, \eps_{l-1}+\eps_{l}\}$ may be identified with the standard set of simple roots $\Pi$ of $\fso(2l)$ given in Section \ref{ss:soevenreal}.  

Now let $\fg=\fso(2l)$ and $\fk=\fso(2l-1)$ and $\theta=\theta_{2l}$ be as above.  
Then the simple roots $\alpha_{l-1}=\epsilon_{l-1}-\epsilon_{l}$ and $\alpha_{l}=\epsilon_{l-1}+\epsilon_{l}$ 
are complex $\theta$-stable with $\theta(\alpha_{l-1})=\alpha_{l}$.  Note that we can choose $\theta$ to be the involution which acts on the basis of $\C^{2l}$ as $\theta(e_{l})=e_{-l}$ and $\theta(e_{\pm i})=e_{\pm i}$ for $i\neq l$.  
Therefore, the roots $\{\pm(\epsilon_{i}+\epsilon_{j}),\,\pm (\epsilon_{i}-\epsilon_{j}),\, 1\leq i<j\leq l-1\}$ are compact imaginary, whereas the roots
$\{\pm(\epsilon_{i}+\epsilon_{l}), \pm(\epsilon_{i}-\epsilon_{l}),\, 1\leq i\leq l-1\}$ are complex $\theta$-stable with $\theta(\epsilon_{i}\pm\epsilon_{l})=\epsilon_{i}\mp\epsilon_{l}.$  
The $\theta$-stable subspace
$\fg_{\alpha}\oplus\fg_{\theta(\alpha)}$ decomposes as $\fg_{\alpha}\oplus\fg_{\theta(\alpha)}=((\fg_{\alpha}\oplus\fg_{\theta(\alpha)})\cap \fk)\oplus((\fg_{\alpha}\oplus\fg_{\theta(\alpha)})\cap\fg^{-\theta}).$

In this case, $\fh\cap\fk$ can be identified with the standard diagonal Cartan subalgebra of $\fk$.  Under this identification, the set $\{\alpha_{1},\,\dots, \, \alpha_{l-2},\, \frac{1}{2}(\alpha_{l-1}+\theta(\alpha_{l-1}))\}$ is identified with the standard simple roots $\Pi$ of $\fso(2l-1)$ given in Section \ref{ss:sooddreal}. 
%Note that for each complex $\theta$-stable root $\alpha$,
\end{exam}

%Let $\fg=\fgl(n,\C)$, and let $\fk=\fgl(n-1,\C)\oplus\fgl(1,\C)$ 
%be the subalgebra consisting of block diagonal matrices with one $(n-1)\times (n-1)$ block 
%and one $1\times 1$ block.  Then $\fk=\fg^{\theta}$, where 
%$\theta=\Ad(t),$ $t=\mbox{diag}[1,\dots, 1,-1]$.  

\subsection{Notation}\label{s:thenotation}
We lay out some of the notation that we will use throughout the paper.   Recall that all Lie algebras and algebraic groups are assumed to be complex.%Sections \ref{ss:realization} and \ref{ss:irredcomp}. 
\begin{nota}\label{nota:thenotation}
\begin{enumerate}
\item We let $\fg= \fso(n)$ unless otherwise specified.  We let 
$\fk=\fso(n-1)$ be the symmetric subalgebra given in Section \ref{ss:symmetricreal}.  
We denote the corresponding connected algebraic subgroup by $K$.  
It will also be convenient at times to denote the Lie algebras $\fso(i)$ by $\fg_{i}$, 
and the algebraic group $SO(i)$ by $G_{i}$ for $i=2,\dots, n$.  In particular,
$\fg_{n-1} = \fk$.
\item We let $r_{i}$ be the rank of $\fg_{i}$.
\item Let $\langle \langle x, \, y\rangle\rangle =\mbox{Tr}(xy)$, for $x,\, y\in\fg$ be the (nondegenerate) trace form on $\fg$.  
\item In the Cartan decomposition, $\fg=\fk + \fp$ is direct, where $\fp = \fg^{-\theta}$.
For $x\in\fg$, we let $x_{\fk}$ and $x_{\fp}$ be the $\fk$ and $\fp$ components.
We recall that $\fp = \fk^{\perp}$, where for $U \subset \fg$,
$U^{\perp}$ is the perpindicular subspace relative to the trace form.
\item  Let $G$ be an algebraic group with Lie algebra $\fg$.  Suppose $\fc\subset\fg$ is an algebraic subalgebra, and $S\subset \fg$ is a subset.  We denote by 
$\fz_{\fc}(S)$ the centralizer of $S$ in $\fc$, i.e.,
$\fz_{\fc}(S):=\{Y\in\fc:\, [Y, s]=0 \mbox{ for all } s\in S\}$.
If $C\subset  G$ is an algebraic subgroup, then we denote the centralizer of $S$ in $C$ by 
$Z_{C}(S)=\{g\in C:\, gsg^{-1}=s\mbox{ for all } s\in S$\}.  Of course $\mbox{Lie}(Z_{C}(S))=\fz_{\fc}(S)$ if $C$ is connected with
Lie algebra $\fc$. 
\item For any complex variety $X$, we denote by $\C[X]$ its ring of regular
functions.  If $X$ is a $G$-variety, we let $\C[X]^G$ be the $G$-invariant
regular functions.  We call an element $x\in X$ regular (or $G$-regular
if $G$ is ambiguous) if $x$ lies in the open dense set consisting of
$G$-orbits of maximal dimension, and let $X_{reg}$ be the regular elements.
If $G$ is a reductive algebraic group acting on its Lie algebra $\fg$ via the adjoint action, we let 
$\C[\fg]^{G}$ be the ring of adjoint invariant polynomial functions on $\fg$,
and recall that $x\in \fg$ is regular if and only if $\fz_{\fg}(x)$
has dimension equal to the rank of $\fg$.
\item We let $\B=\B_{\fr}$ denote the flag variety of a Lie algebra $\fr$,
which we identify with the Borel subalgebras of $\fr$.
\end{enumerate}
\end{nota}  

\begin{dfn}\label{dfn:std}
We shall call a Borel subalgebra $\fb\in \B$ \emph{standard} if $\fh\subset\fb$, where 
$\fh\subset\fg$ is the standard Cartan subalgebra of diagonal matrices of $\fg$.

\end{dfn}

\subsection{Numerical identities}\label{ss:numerical}

We record several numerical identities which are used in the remainder
of the paper.   Let $\fg = \fso(n)$ and $\fk = \fso(n-1)$.

\begin{equation}\label{eq:sumri}
\displaystyle\sum_{i=2}^{n-1} r_{i} = \frac{1}{2}(\dim\fg-r_{n}).
\end{equation}

\begin{equation}\label{eq:dimbigflag}
\dim(\B_{\fg}) + \dim(\B_{\fk})=\dim(\fg)-r_{n}-r_{n-1}.
\end{equation}

\begin{equation}\label{eq:quotdim}
\dim(\fg)-r_{n}-r_{n-1}=\dim \fk.
\end{equation}

These assertions are routine and are left to the reader.  We note
that they are also true for $\fgl(n)$, provided the sum in the first
equation goes from $1$ to $n-1$.

\subsection{The Gelfand-Zeitlin Integrable Systems}\label{ss:GZsystems}

The complex general linear GZ system was first introduced 
by Kostant and Wallach in \cite{KW1}, and the orthogonal GZ system was introduced by the first author in \cite{Col2}.  
We briefly recall the construction of the orthogonal GZ systems here, especially since 
we are using a different realization of the orthogonal Lie algebra than was used in 
\cite{Col2}.  To construct the GZ system, let $\fg_{i} \cong \fso(i)$ be defined by downward induction,
by taking $\fg_{n}=\fg$, and letting $\fg_{i}=\fg_{i+1}^{\theta_{i+1}}$, where
$\theta_{i+1}$ is the involution from Section \ref{ss:symmetricreal}.
Thus, 
we have a chain of Lie subalgebras 
\begin{equation}\label{eq:GZchain}
\fg_{2}\subset\fg_{3}\subset\dots\subset\fg_{i}\subset\dots\subset \fg_{n}.
\end{equation} 
The Lie algebra $\fg$ 
has nondegenerate, invariant, symmetric bilinear form $\langle\langle x, y\rangle\rangle$ given by the trace form from Notation \ref{nota:thenotation} (3).
The form is nondegenerate on each $\fg_{i}$, so
we may identify $\fg_{i}$ with its dual, and regard it as a Poisson
variety using the Lie-Poisson structure. 
For $i=2, \dots, n$, the inclusion $\fg_{i}\subset\fg$ then dualizes 
to give us a map: $\pi_{i}:\fg\to\fg_{i}$, which projects an element 
$x\in\fg$ to its projection $x_{i}$ in $\fg_{i}$
off of $\fg_{i}^{\perp}$. 
Define functions $\psi_{i,j}$ on $\fg_{i}$ for $i=2, \dots, n$ as follows.
\begin{equation}\label{eq:generators}
\begin{split}
&\mbox{ If } i \mbox{ is odd, then}\, \psi_{i,j}(y):=\mbox{Tr}(y^{2j}), \, y\in\fg_{i},\, j=1,\dots, r_{i}.\\
&\mbox{ If } i\mbox{ is even, then}\, \psi_{i,j}(y):=\mbox{Tr}(y^{2j}),\, y\in\fg_{i},\,\mbox{for } j=1,\dots, r_{i}-1,\mbox{ and } \psi_{i,r_{i}}(y):=\mbox{Pfaff}(y), \\
&\mbox{ where } \mbox{Pfaff}(y) \mbox{ denotes the Pfaffian of } y.\\
\end{split}
\end{equation}    
Then it is well-known that for each $i$, $\C[\fg_{i}]^{G_{i}}$ is a polynomial algebra with
free generators $\psi_{i,j}, j=1, \dots, r_{i}$, and $\C[\fg_{i}]^{G_{i}}$ Poisson
commutes with $\C[\fg_{i}]$. 
 Since $\fg_{i}\subset\fg$ is an inclusion of Lie algberas, 
the transpose $\pi_{i}:\fg\to\fg_{i}$ is easily seen to be a map of Poisson varieties.   For $i=2,\dots, n$ and $j=1,\dots, r_{i}$, let $f_{i,j}=\pi_{i}^{*}\psi_{i,j}$.   
We define the Gelfand-Zeitlin functions: 
\begin{equation}\label{eq:GZfuns}
J_{GZ}=\{f_{i,j}:\; i=2,\dots, n,\, j=1,\dots, r_{i}\},
\end{equation}
and let $J(\fg)$ be the subalgebra of $\C[\fg]$ generated by the
functions in $J_{GZ}$.
\begin{rem}\label{r:generallinear}
The Gelfand-Zeitlin functions $J_{GZ}$ and the associated subalgebra $J(\fg)$ were first considered for the $n\times n$ complex general linear Lie algebra $\fg=\fgl(n)$ by 
Kostant and Wallach in \cite{KW1}.  In that setting, there is also a chain of subalgebras as in (\ref{eq:GZchain}), 
where $\fg_{i}=\fgl(i)\subset\fg$ is identified with the top lefthand $i\times i$ corner of $\fg$.  See \cite{KW1} or \cite{CEexp} for more details.
\end{rem}

The following routine proposition may be proved by the same method as for
 the general linear case
in Proposition 2.5 of \cite{CEexp}. 

\begin{prop}\label{p:Poissoncommute}
The algebra $J(\fg)$ is a Poisson commutative subalgebra of $\C[\fg]$.
\end{prop}

Our main goals in this paper are to show that the restriction of the functions $J_{GZ}$ to 
any regular $\Ad(G)$-orbit in $\fg$ forms a completely integrable system and to understand the generic leaves of the foliation 
given by the integrable system.  For both of these issues, we need to study the moment map for the system, which we call
 the (orthogonal) \emph{Kostant-Wallach map} or KW map for short.  This 
is the morphism $\Phi:\fg\to\C^{r_{2}}\times\dots\times\C^{r_{n}}$ given by:
\begin{equation}\label{eq:KWmap}
\Phi:\fg\to\C^{r_{2}}\times\dots\times\C^{r_{n}};\; \Phi(x):=(f_{2,1}(x), \dots, f_{i,1}(x),\dots, f_{i, r_{i}}(x),\dots, f_{n,1}(x),\dots, f_{n, r_{n}}(x)).
\end{equation}

\begin{nota}\label{nota:spectrum}
For $x\in\fg$, let $\sigma(x)$ denote the spectrum of $x$.   
If $\fg$ is type $B$, then zero occurs as an eigenvalue of $x$ with multiplicity at least one.  
In this case, we only consider $0\in\sigma(x)$, if it occurs as an eigenvalue 
of $x$ with multiplicity strictly greater than one.
\end{nota}

\begin{rem}\label{r:spectrum}
We observe that if $y\in\Phi^{-1}(\Phi(x))$, then $\sigma(x_{i})=\sigma(y_{i})$ for all $i=2,\dots, n$.  
This follows from the well-known fact that the values of $x_{i}$ on the basic adjoint invariants $\psi_{i,1},\dots, \psi_{i,r_{i}}$ 
in (\ref{eq:generators}) determine the characteristic polynomial of $x_{i}$. 
\end{rem}

Elements of $x\in\fg$ which lie in a regular level set of $\Phi$ play a very important role in the study of the GZ systems. 
\begin{dfn-nota}\label{dfnote:sreg}
An element $x\in\fg$ is said to be \emph{strongly regular }if the differentials 
\begin{equation}\label{eq:GZdiffs}
\{df(x): f\in J_{GZ}\}\subset T_{x}^{*}(\fg)
\end{equation}
are linearly independent elements of the cotangent space $T_{x}^{*}(\fg)$.  We denote
the set of strongly regular elements as $\fg_{sreg}$ and note that $\fg_{sreg}\subset\fg$ is Zariski open.  
For $x\in \fg$, we set $\Phi^{-1}(\Phi(x))_{sreg}:=\Phi^{-1}(\Phi(x))\cap\fg_{sreg}.$  
\end{dfn-nota}
The term strongly regular is suggested by the following classical result of Kostant.  
Let $\fg$ be a complex reductive Lie algebra with adjoint action, and let
$\phi_1, \dots, \phi_r$ be generators of the polynomial algebra $\C[\fg]^G$.
Then (Theorem 9, \cite{Kostant63}) 
\begin{equation}\label{eq:regdiffs1}
x\in\fg_{reg} \mbox{ if and only if } d\phi_{1}(x)\wedge\dots\wedge d\phi_{\ell}(x)\neq 0. 
\end{equation}
Hence, by Equation (\ref{eq:GZfuns}), we see that
\begin{equation}\label{eq:totreg}
\mbox{If } x\in\fg_{sreg},\mbox{ then }\, x_{i}\in\fg_{i} \mbox{ is regular for all } i=2,\dots, n. 
\end{equation}

We briefly recall the notion of an integrable system from symplectic geometry. 
Recall that if $M$ is a complex manifold then the holomorphic functions 
$\{F_{1},\dots, F_{r}\}$ are said to be independent on $m$ if the open subset 
$$
U=\{m\in M: \; dF_{1}(m)\wedge\dots\wedge dF_{r}(m)\neq 0\}
$$
is dense in $M$.  
\begin{dfn}\label{dfn:integrable}
Let $(M,\omega)$ be a complex symplectic manifold of dimension $2r$.  An \emph{integrable system} on $M$ is a collection
of $r$ independent, holomorphic functions $\{F_{1},\dots, F_{r}\}$ which Poisson commute with respect to the natural Poisson bracket 
on the space of holomorphic functions defined by the symplectic form $\omega$. 
\end{dfn}
We recall that for any $x\in\fg$, the adjoint orbit through $x$, $\Ad(G)\cdot x$ is a symplectic manifold with Kostant-Kirillov-Souriau symplectic structure.
  The connection between the complete integrability of the GZ system on regular adjoint orbits and 
strongly regular elements is given by the following proposition.  

\begin{prop}\label{prop:sregintegrable}
Let $x\in\fg_{reg}$, and let $\Ad(G)\cdot x$ be the adjoint orbit of $G$ through $x$.  Then the restriction 
of the GZ functions $J_{GZ}$ in (\ref{eq:GZfuns}) form a completely integrable system on $\Ad(G)\cdot x$ if and only 
if 
\begin{equation}\label{eq:sregintersect}
\Ad(G)\cdot x\cap\fg_{sreg}\neq\emptyset.
\end{equation}
\end{prop}
\begin{proof}
Consider the GZ functions $J_{GZ}$ and let 
$$
T = \{ f_{i,j} \mbox{ for } i=2,\dots,n -1, j=1,\dots, r_{i}\},
$$
and let $S = \{ f|_{\Ad(G)\cdot x} :  f \in T \}$.
Since the functions $f_{n,j}$, $j=1,\dots, r_{n}$, restrict to constant functions on $\Ad(G)\cdot x$, we do not include them in $T$.
Let $x\in\fg_{reg}$ and let $z\in \Ad(G)\cdot x\cap \fg_{sreg}\neq\emptyset.$
Let $\tilde{R}$ be the span of the differentials
$\{ df(z) : f\in T \} \subset T_z^*(\fg)$,
and let $R$ be the restriction of $\tilde{R}$ to $T_z(\Ad(G)\cdot x)$.
If $\dim(R) < |S|$, there is a non-zero linear functional
$\lambda \in \tilde{R} \cap  T_z(\Ad(G)\cdot x)^{\perp}$.  But this contradicts
the fact that $z\in\fg_{sreg}$, since the trace form identifies $T_z(\Ad(G)\cdot z)^{\perp}$ with the span
of $df_{n,j}(z), j=1, \dots, r_n$, and hence $\dim(R) = |S|$.  The
complete integrability of $S$ on $\Ad(G) \cdot x$ now follows from
standard assertions, Proposition \ref{p:Poissoncommute},
 and Equation (\ref{eq:sumri}).  We leave the converse
assertion to the reader.
\end{proof}
For $\fg=\fgl(n)$, Kostant and Wallach introduced the GZ integrable 
system using a chain of subalgebras analogous to the one we used in (\ref{eq:GZchain}) (Remark \ref{r:generallinear}). 
The definitions of the Kostant-Wallach map and strong regularity are the same.  Abusing notation, we will also refer to Kostant-Wallach map for 
$\fgl(n)$ as $\Phi$.  In the general linear case, Kostant and Wallach show that Equation (\ref{eq:sregintersect}) is satisfied for every $x\in \fgl(n)_{reg}$ (see Theorem 3.36 of \cite{KW1}).  
Their proof makes use of the fact that in this case $\Phi$ possesses a natural cross-section given by the so-called
upper Hessenberg matrices.  
These are matrices of the form:
\begin{equation}\label{eq:Hessenberg}
Hess:=\left [\begin{array}{ccccc}
a_{11} & a_{12} &\cdots & a_{1n-1} & a_{1n}\\
1 & a_{22} &\cdots & a_{2n-1} & a_{2n}\\
0 & 1 & \cdots & a_{3n-1}& a_{3n}\\
\vdots &\vdots &\ddots &\vdots &\vdots\\
0 & 0 &\cdots & 1 &a_{nn}\end{array}\right ]_{,}
\end{equation}
with $a_{ij}\in\C$.
In Theorem 2.3 of \cite{KW1}, the authors prove that the restriction of $\Phi$ to $Hess$ is an isomorphism of varieties from 
which it follows that 
\begin{equation}\label{eq:sregfibre}
\Phi^{-1}(\Phi(x))_{sreg}\neq\emptyset\mbox{ for any } x\in\fgl(n),
\end{equation}
and (\ref{eq:sregintersect}) follows easily.
  It is not clear in the orthogonal case whether $\Phi$ is even surjective nor that every non-empty fibre
contains strongly regular elements.

 \section{The Partial Kostant-Wallach Map and Surjectivity of $\Phi$}\label{s:surjectivity}

Rather than considering all of the Lie algebras in the chain in (\ref{eq:GZchain}) simultaneously and working 
directly with the KW map, it is easier to consider one step in the chain at a time: $\fg_{n-1}\subset\fg$.  Accordingly, 
we define the \emph{partial Kostant-Wallach map} to be 
\begin{equation}\label{eq:partial}
\begin{array}{c}
\Phi_{n}:\fg\to \C^{r_{n-1}}\oplus \C^{r_{n}}, \\
\\
\; \Phi_n(x)=(f_{n-1, 1}(x),\dots, f_{n-1, r_{n-1}}(x), f_{n,1}(x),\dots, f_{n, r_{n}}(x)).
\end{array}
\end{equation}
%where $\C[\fg_{i}]^{G_{i}}=\C[f_{i,1},\dots, f_{i, r_{i}}]$.
The map $\Phi_{n}$ has the advantage that it is a geometric invariant theory quotient (GIT quotient) 
as well as having other remarkable properties.  To see these properties of $\Phi_{n}$,  we will need the theory
of spherical pairs.
% partial KW map $\Phi_{n}$ has some remarkable properties.

\subsection{Spherical pairs and their coisotropy representations}\label{ss:sphercoisotropy}
In this section, we consider pairs $(M, H)$ where $M$ is a reductive algebraic group and $H\subset M$ is an algebraic subgroup.  
Let $\fm$ and $\fh$ be the Lie algebras of $M$ and $H$ respectively.  

%Consider a pair $(M,H)$ where $H$ is an algebraic subgroup of a reductive
%algebraic group $M$.   The pair is called reductive if $H$ is reductive.

\begin{dfn}\label{dfn:spherical}

(1)  The pair $(M, H)$ is called spherical if $H$ acts on the flag variety $\B = \B_{\fm}$ 
with finitely many orbits. 

\noindent (2) The pair $(M,H)$ is called reductive if $H\subset M$ is a reductive algebraic subgroup.  
\end{dfn}

\begin{dfn}\label{dfn:multfree}
For a reductive pair $(M,H)$ we say the branching law from $M$ to $H$
is multiplicity free if for all irreducible rational representations
$V$ and $U$ of $M$ and $H$ respectively, $\dim(\Hom_H(U,V)) \le 1$.
\end{dfn}

For a pair $(M_{1}, R)$, we define a new pair $(M, H)$ by taking 
 $M={\tilde{M}}_1 := M_1 \times R$ and taking $H = R_{\Delta} = \{ (g, g): g\in R\} \subset {\tilde{M}}_1$.   Let ${\tilde{\fm}}_1 = \fm_1 \oplus \fr$ and $\fr_{\Delta}$ be the corresponding Lie algebras.  The following result is well-known.
\begin{prop}\label{prop:isspherical}
\begin{enumerate}
\item Let $(M_1, R)$ be a reductive pair.
The pair $(\tilde{M_1}, R_{\Delta})$ is spherical 
if and only if the branching rule from $M_1$ to $R$ is multiplicity free. 
\item For the pair $(M_1,R)=(SO(n), SO(n-1))$, the pair 
$(\tilde{M_{1}}, R_{\Delta})=(SO(n)\times SO(n-1), SO(n-1)_{\Delta})$ is spherical.
\end{enumerate}
\end{prop}
\begin{proof}
The first statement follows by Theorem B of \cite{brundanspher}, together
with the easy observaton that a Borel subgroup $B_R$ of $R$ has
finitely many orbits on the flag variety $\B_{\fm_{1}}$ of $\fm_{1}$ if and only if
$R_{\Delta}$ has finitely many orbits on $\B_{\fm_{1}} \times \B_{\fr}$.   The
second statement follows from the first statement and well-known
branching laws (see \cite{Johnson}).
\end{proof}

For a reductive spherical pair $(M, H)$, let $\langle\langle\cdot, \cdot\rangle\rangle$ be a non-degenerate symmetric $M$-invariant bilinear form on 
$\fm$, and let $\fh^{\perp}$ be the annihilator of $\fh$ with respect to $\langle\langle\cdot, \cdot\rangle\rangle$.  Then the adjoint action of 
$M$ on $\fm$ restricts to an action of $H$ on $\fh^{\perp}$, which is referred to in the literature 
as the coisotropy representation of $H$ (see \cite{Pancoiso}).  Then it is well-known that $\C[\fh^{\perp}]^{H}$ is a polynomial algebra (Kor 7.2 of \cite{Kn} or Corollary 5 of \cite{Pancoiso}).

\subsection{Surjectivity of $\Phi_n$ and $\Phi$}\label{ss:partialtotalsurj}

We now return to our convention with $G=SO(n)$ and $K=SO(n-1)$, and
similarly with Lie algebras.  We consider the bilinear form
on $\tilde{\fg} = \fg + \fk$ given by taking the trace form on each
factor, and note that by an easy calculation,
\begin{equation}\label{e:productperp}
\fk_{\Delta}^{\perp}=\{(x, -x_{\fk}): x\in\fg, \, x_{\fk}\in \fk\} \cong \fg
\end{equation}
as a $K \cong K_{\Delta}$-representation.

\begin{prop}\label{prop:flat}
\begin{enumerate}

\item $\C[\fg]^K=\C[\fg]^G \otimes \C[\fk]^K$.
\item $\Phi_n$ coincides with the invariant theory quotient
morphism $\fg \to \fg//K.$   In particular, $\Phi_n$ is
surjective.
\item The morphism $\Phi_{n}$ is flat.  In particular, its fibres
are equidimensional varieties of dimension $\dim \fg- r_{n}-r_{n-1}$. 
\end{enumerate}
\end{prop}

\begin{proof} Recall the well-known fact that the fixed point algebra $U(\fg)^K$
of $K$ in the enveloping algebra $U(\fg)$ is commutative \cite{Johnson}.   Hence, $U(\fg)^K$ coincides with its centre,
$Z(U(\fg)^K)$.   In Theorem 10.1 of \cite{Knhc}, Knop shows that
$Z(U(\fg)^K) \cong
U(\fg)^G \otimes_{\C} U(\fg)^K$.  The first assertion now follows by
taking the associated graded algebra with respect to the usual filtration
of $U(\fg)$.  By the first assertion, $\Phi_n$ coincides with the
invariant theory quotient $\fg \to \fg//K$, which gives the second
assertion.   By Equation (\ref{e:productperp}), we identify 
$\fk_{\Delta}^{\perp} \cong \fg$ $K$-equivariantly.    
  Then the flatness
of $\Phi_n$ follows by Korollar 7.2 of \cite{Kn}, which gives a criterion
for flatness of invariant theory quotients in the setting of
 spherical homogeneous spaces (see also \cite{Pancoiso}).
\end{proof}

\begin{nota}\label{nota:partialnil}
For ease of notation, we denote the nilfibre $\Phi_{n}^{-1}(0,0),\, (0,0)\in\C^{r_{n-1}}\times\C^{r_{n}}$ of $\Phi_{n}$ by $\Phi_{n}^{-1}(0)$.
\end{nota}

Using the flatness of the partial KW map $\Phi_{n}$, we can now show that 
the orthogonal KW map $\Phi$ is surjective.

\begin{thm}\label{thm:surj}
Let $\Phi:\fg\to \C^{r_{2}}\times\dots\times \C^{r_{n}}$ be the
the Kostant-Wallach map.  The morphism $\Phi$ is surjective and every fibre of $\Phi$ 
contains a regular element of $\fg$.
\end{thm}

To prove Theorem \ref{thm:surj}, we need some preparation.  
Consider the nonempty Zariski open set:
\begin{equation}\label{eq:bothopenU}
U_{r,\fk}:=\{x\in \fg_{reg}:\, x_{\fk} \in\fk_{reg}\}.
\end{equation}

%Note that $U\neq \emptyset$, since $\fg_{sreg}\neq 0$ by Theorem 3.2, \cite{Col2} and (\ref{eq:totreg}).  

\begin{lem}\label{l:bothreg}
Let $\Phi_{n}:\fg\to \C^{r_{n-1}}\times \C^{r_{n}}$ be the \emph{partial} KW map. 
The restriction $\Phi_{n}|_{U_{r,\fk}}: U_{r,\fk}\to \C^{r_{n-1}}\times\C^{r_{n}}$ is surjective.
\end{lem}
\begin{proof}
Since $\Phi_{n}$ is a flat morphism, it is an open morphism by Exercise III.9.1
of \cite{Ha}.  
Thus, $\Phi_{n}(U_{r,\fk})\subseteq \C^{r_{n-1}}\times \C^{r_{n}}$ is  Zariski open.  
We suppose that $\Phi_{n}(U_{r,\fk})\neq \C^{r_{n-1}}\times \C^{r_{n}}.$  
Then $\mathcal{C}:=\C^{r_{n-1}}\times\C^{r_{n}}\setminus \Phi_{n}(U_{r,\fk})$ is a non-empty, closed subset 
of $\C^{r_{n-1}}\times \C^{r_{n}}$.  Since $\Phi_{n}$ is surjective, it follows that $\Phi_{n}^{-1}(\mathcal{C})$ is a non-empty,  closed subset of $\fg$.  
Now it follows from definitions that
\begin{equation}\label{eq:Cdfn}
x\in\Phi_{n}^{-1}(\mathcal{C})\mbox{ if and only if } \Phi_{n}^{-1}(\Phi_{n}(x))\cap U_{r,\fk}=\emptyset.
\end{equation}
Since the functions defining $\Phi_{n}$ are homogeneous (Equation (\ref{eq:partial})), scalar multiplication by 
$\lambda \in \C^{\times}$ induces an isomorphism:  $\Phi_{n}^{-1}(\Phi_{n}(x))\to \Phi_{n}^{-1}(\Phi_{n}(\lambda x))$.  
Note that $U_{r,\fk}$ is also preserved by scalar multiplication by elements of $\C^{\times}$.  Thus, (\ref{eq:Cdfn}) implies 
that $\Phi_{n}^{-1}(\mathcal{C})$ is preserved by scalar multiplication by elements of $\C^{\times}$.  
Since $\Phi_{n}^{-1}(\mathcal{C})$ is closed, it follows that $0\in \Phi_{n}^{-1}(\mathcal{C})$, whence 
\begin{equation}\label{eq:empty}
\Phi_{n}^{-1}(0)\cap U_{r,\fk}=\emptyset.
\end{equation}
But we claim that $\Phi_{n}^{-1}(0)\cap U_{r,\fk}\neq \emptyset$, contradicting the initial assumption that $\Phi_{n}(U_{r,\fk})\neq \C^{r_{n-1}}\times\C^{r_{n}}$.

For this,  let $\fb_{+}\subset \fg$ be the Borel subalgebra of upper triangular matrices in $\fg$, and let $\fn_{+}=[\fb_{+},\fb_{+}]$.  
We consider the cases where $\fg$ is type $B$ and $D$ separately.  If $\fg=\fso(2l+1)$ is type $B$, we let 
$e:=e_{\alpha_{1}}+\dots +e_{\alpha_{l}}+e_{\eps_{l-1}+\eps_{l}}$, where $e_{\alpha_{i}}\in \fg_{\alpha_{i}}, \, i=1,\dots, l$ and $e_{\eps_{l-1}+\eps_{l}}\in\fg_{\eps_{l-1}+\eps_{l}}$ are nonzero root vectors. 
Then by Kostant's criterion for an element to be regular nilpotent
(Theorem 5.3 of \cite{Kostanttds}),
$e\in\fn_{+}$ is regular nilpotent and $\pi_{\fk}(e)=e_{\alpha_{1}}+\dots+e_{\alpha_{l-1}}+e_{\eps_{l-1}+\eps_{l}}$ by Example \ref{ex:roottypes}.  By Kostant's criterion,
the element $\pi_{\fk}(e)$ is a regular nilpotent element of $\fk$.  Thus, $e\in\Phi_{n}^{-1}(0)\cap U_{r,\fk}$.   

Now let $\fg=\fso(2l)$ be type $D$.  Let $e=e_{\alpha_{1}}+\dots +e_{\alpha_{l-2}}+ \frac{1}{2}(e_{\alpha_{l-1}}+\theta(e_{\alpha_{l-1}}))$.  
By Example \ref{ex:roottypes}, we have $\theta(\alpha_{l-1})=\alpha_{l}$, so $e$ is regular nilpotent by Kostant's criterion.  Example \ref{ex:roottypes} also implies
that $e\in\fk$, and that in terms of the simple roots for $\fk=\fso(2l-1)$, we have $e=e_{\alpha_{1}}+\dots +e_{\alpha_{l-2}}+e_{\alpha_{l-1}}$.  Therefore, again by Kostant's criterion,
$e$ is also regular nilpotent when viewed as an element of $\fk$.  Thus, $e\in \Phi_{n}^{-1}(0)\cap U_{r,\fk}$.  

We conclude that $\Phi_{n}^{-1}(0)\cap U_{r,\fk}\neq \emptyset$, contradicting (\ref{eq:empty}) and the initial assumption that $\Phi_{n}|_{U_{r,\fk}}$ is not surjective. 

\end{proof}

\begin{proof}[Proof of Theorem \ref{thm:surj}]
The proof proceeds by induction on $n$.  The case $n=3$, follows from the fact that $\fso(3)\cong\fsl(2)$, and the theorem is known to hold 
for $\fsl(2)$.  We now assume that the result is true for $\fk$.  Let $(c_{2},\dots, c_{n-1}, c_{n})\in \C^{r_{2}}\times \dots \times \C^{r_{n-1}}\times \C^{r_{n}}$ be arbitrary.  Consider $\Phi^{-1}(c_{2},\dots, c_{n-1}, c_{n})$.  Let $\Phi_{\fk}: \fk\to \C^{r_{2}}\times \dots\times \C^{r_{n-1}}$ be the KW map 
for $\fk$.  By the inductive hypothesis, there exists an element $y\in \Phi_{\fk}^{-1}(c_{2},\dots, c_{n-1})$ with $y\in\fk_{reg}$.  
By Lemma \ref{l:bothreg}, there exists $x\in\Phi_{n}^{-1}(c_{n-1},c_{n})$ with $x\in\fg_{reg}$ and $x_{\fk}\in\fk_{reg}$.  
It follows that $x_{\fk},\, y\in \chi_{\fk}^{-1}(c_{n-1})\cap\fk_{reg}$, where $\chi_{\fk}:\fk\to \fk//K$ is the adjoint quotient, 
so that $y=\Ad(k)\cdot x_{\fk}$ for some $k\in K$ by Theorem 3 of \cite{Kostant63}.  Let $z:=\Ad(k)\cdot x$.  Then $z\in\Phi_{n}^{-1}(c_{n-1},c_{n})$ is regular, and $z_{\fk}=\Ad(k)\cdot x_{\fk}=y$.  It follows 
from definitions that $z\in \Phi^{-1}(c_{2},\dots, c_{n-1}, c_{n})$.  This completes the proof of the theorem.
\end{proof}

\section{Spherical Pairs and strongly regular elements}\label{s:spherical}

In this section, we use the geometry of spherical varieties to study the set of strongly regular elements of $\fg_{sreg}$ of $\fg$ introduced in Definition-Notation \ref{dfnote:sreg}.  
We simplify the criterion of Kostant and Wallach for an element $x\in \fg$ to be strongly regular and use our new characterization of strongly regular elements to 
construct a new set of strongly regular elements $\fg_{\Theta}$ for $\fg$.  As a consequence, we obtain the complete integrability 
of the GZ system on regular adjoint orbits on $\fg$ and determine the structure of the KW fibres $\Phi^{-1}(\Phi(x))$ for $x\in\fg_{\Theta}$.  We show that each such fibre 
has a free action of an abelian algebraic group thereby obtaining ``angle coordinates" for the GZ system on these fibres.  
%and obtain ``angle coordinates" 
%for the integrable system on these fibres using a free action of an abelian algebraic group. 

\subsection{Further results on the coisotropy representation of a spherical pair}\label{ss:morecoisotropy}
Recall the reductive spherical pair $(M,H)$ from Definition \ref{dfn:spherical}.  We consider reductive spherical pairs satisfying: 
\begin{equation}\label{eq:numerology}
\dim\B=\dim \fh^{\perp}-\dim \fh^{\perp}// H.  
\end{equation}

%THE NOTATION HERE IS A BIT INCONSISTENT $\B$ or $\B_{\fm}$.
\begin{rem}\label{r:appendix}
If we let $G=GL(n)$, and $K=GL(n-1)$, where $K$ is embedded in $G$ in the top lefthand 
corner (Remark \ref{r:generallinear}), then the pair $(\tilde{G}, K_{\Delta})$ is spherical 
by Part (1) of Proposition \ref{prop:isspherical}.  Further, this pair satisfies Equation (\ref{eq:numerology}).  
In Theorem 2.3, \cite{CEeigen} we gave a different proof of Knop's flatness result (Korollar 7.2, \cite{Kn}) for the spherical
pair $(GL(n)\times GL(n-1), GL(n-1)_{\Delta})$ using a variant of the Steinberg variety along with a dimension estimate obtained from (\ref{eq:numerology}).  Our proof can be generalized to any reductive spherical pair $(M,H)$ satisfying Equation (\ref{eq:numerology}).  See the appendix of \cite{CEspherical} for details. 
\end{rem}

We now analyze further the meaning of the condition in Equation
 (\ref{eq:numerology}) 
for the coisotropy representation.  
We study the set of $H$-regular elements $\fh_{reg}^{\perp}$ in $\fh^{\perp}$
consisting of $H$-orbits of maximal dimension, and recall the regular set
$\fm_{reg}$ of $\fm$ (Notation \ref{nota:thenotation} (6)).

\begin{thm}\label{thm:regelts}
Let $(M, H)$ be a reductive spherical pair.  Then the following conditions are equivalent.
\begin{enumerate}
\item Equation (\ref{eq:numerology}) holds.
\item We have $\fh^{\perp}_{reg}\subset \fm_{reg}$. 
\end{enumerate}
\end{thm}

\begin{proof}
By Theorems 3 and 6 and Equation (15) 
of \cite{Pancoiso}, there is an open dense set $U$ of $\fh^{\perp}_{reg}$ such that for $y\in U$,  
\begin{equation}\label{eq:general1}
\codim_{\fh^{\perp}} (\Ad(H)\cdot y)=\dim (\fh^{\perp}//H),
\end{equation}
and 
\begin{equation}\label{eq:general2} 
\dim(\Ad(M)\cdot y)=2\dim(\Ad(H)\cdot y).
\end{equation}
We consider $x\in\fh_{reg}^{\perp}$.  By (\ref{eq:general1}), we have 
$\codim_{\fh^{\perp}} (\Ad(H)\cdot x)=\dim (\fh^{\perp}//H)$.  Assuming
(1), we see that 
\begin{equation}\label{eq:firstdimension}
\dim (\Ad(H)\cdot x)=\dim \B.
\end{equation}
 By Proposition (1)
of \cite{Pancoiso}, $2\dim(\Ad(H)\cdot x) \le \dim(\Ad(M)\cdot x)$ so that
$\dim(\Ad(M)\cdot x) \ge 2 \dim(\B)$, and thus $x \in \fm_{reg}$, which
proves one direction of the assertion.  

Conversely, for any element $x$ in the open set $U\subset \fh_{reg}^{\perp}$, then $x\in\fm_{reg}$ by assumption.  Hence by 
(\ref{eq:general2}), $\dim(\Ad(H)\cdot x)=\frac{1}{2}\dim(\Ad(M)\cdot x)=\dim(\B)$.  The assertion of (1) now follows from Equation (\ref{eq:general1}).

\end{proof}

\begin{rem}\label{r:symmnumerical}
If $\theta$ is an involution of $\fm$ with fixed subalgebra $\fh=\fm^{\theta}$,
the involution is called quasi-split if $\fm$ has a  Borel
subalgebra $\fb$ such that $\fb \cap \theta(\fb)$ is a Cartan subalgebra of 
$\fm$.   In Proposition 4.4 of \cite{CEspherical}, we show that for
$M$ and $H$ the corresponding connected groups, $(M,H)$ satisfies 
Equation (\ref{eq:numerology}) if and only if $\theta$ is quasi-split.
\end{rem}

\subsection{A new criterion for strong regularity}
We now apply Theorem \ref{thm:regelts} to study the $K$-action on 
$\fg$ in the case where $(K, \fg)=(SO(n-1), \fso(n))$ and prove 
an analogue of Kostant's Theorem (\ref{eq:regdiffs1}) for the action of $K$ on $\fg$ by restricting 
the adjoint action of $G$ to $K$ (Theorem \ref{thm:Kostant}). 
For this, we consider the reductive spherical pair $(\tilde{G}, K_{\Delta})$ with
$\tilde{G}=G\times K$.  Recall that we identify $\fk_{\Delta}^{\perp}\cong\fg$ as a $K_{\Delta}\cong K$-module (Equation (\ref{e:productperp})).  

\begin{lem} \label{lem:dimest} Consider the reductive spherical pair $(\tilde{G}, K_{\Delta})$.
\begin{enumerate} \item  Equation (\ref{eq:numerology}) holds.
\item $(\fk_{\Delta}^{\perp})_{reg}\cong\{x\in\fg:\, \dim(\Ad(K)\cdot x)=\dim K\}. $ 
\item \begin{equation}\label{eq:Kreg}
(\fk_{\Delta}^{\perp})_{reg}\cong\{x\in\fg: \fz_{\fk}(x_{\fk})\cap\fz_{\fg}(x)=0\}.
\end{equation} 
\end{enumerate}
\end{lem}

\begin{proof}
Equation (\ref{eq:numerology}) is equivalent to the 
identity 
\begin{equation}\label{eq:flagvarieties}
\dim(\B_{\fg}) + \dim(\B_{\fk})=\dim(\fg)-r_{n}-r_{n-1} =\dim(\fg)-\dim(\fg//K),
\end{equation}
which follows by Equation (\ref{eq:dimbigflag}) and Proposition \ref{prop:flat} (2).  
To prove the second assertion, let 
$x\in (\fk_{\Delta})^{\perp}_{reg}$.  By (1), we can
apply Equation (\ref{eq:firstdimension}) to conclude that $\dim (\Ad(K)\cdot x)=\dim (\B_{\fg})+\dim (\B_{\fk}).$  
The assertion now follows from (\ref{eq:flagvarieties}) and Equation (\ref{eq:quotdim}).  
The second assertion implies that $x\in (\fk_{\Delta}^{\perp})_{reg}$ if and only if $\fz_{\fk}(x)=0$.  The third 
assertion follows since $\fz_{\fk}(x)=\fz_{\fk}(x_{\fk})\cap\fz_{\fg}(x)$. 
\end{proof}

We now describe the regular elements 
of the coisotropy representation of the reductive spherical pair $(\tilde{G}, K_{\Delta})$, which establishes an analogue of
Kostant's theorem.  Recall that an element 
$x\in\fg_{reg}$ if $\dim\fz_{\fg}(x)=\mbox{rank}(\fg)$.  If we identify $T_{x}^{*}(\fg)$ with $\fg$ using 
 the non-degenerate form on $\fg$, then Kostant's basic result in (\ref{eq:regdiffs1}) implies that 
 \begin{equation}\label{eq:centralizer}
 \mbox{span}\{ df_{n,i}(x):\, i=1,\,\dots, r_{n}\}=\fz_{\fg}(x),
 \end{equation}
 Recall also that $\C[\fg]^{K}=\C[f_{n-1,1},\dots, f_{n-1, r_{n-1}}, f_{n,1},\dots, f_{n,r_{n}}]$ (see Proposition \ref{prop:flat}).  
 Let 
$$
\omega_{\fg//K}:=df_{n-1,1} \wedge \dots \wedge df_{n-1, r_{n-1}} \wedge df_{n, 1} \wedge 
\dots \wedge df_{n, r_n} \in \Omega^{r_{n-1} + r_n}(\fg).
$$

\begin{thm}\label{thm:Kostant}
$x\in(\fk_{\Delta}^{\perp})_{reg}\mbox{ if and only if } 
\omega_{\fg//K}(x)\neq 0,$ and if so, then $x\in\fg_{reg}$ and $x_{\fk}\in\fk_{reg}$.
\end{thm}
\begin{proof}
We first suppose that $\omega_{\fg//K}(x)\neq 0$.  By Equation (\ref{eq:regdiffs1}), it 
follows that $x\in\fg_{reg}$ and $x_{\fk}\in\fk_{reg}$.  
Equation (\ref{eq:centralizer}) then implies that 
$\fz_{\fk}(x_{\fk})\cap\fz_{\fg}(x)=0$, so $x\in (\fk_{\Delta}^{\perp})_{reg}$ by Equation (\ref{eq:Kreg}).

Conversely, suppose $x\in(\fk_{\Delta}^{\perp})_{reg}$.  Then by Theorem \ref{thm:regelts}, Equation (\ref{e:productperp}), and part (1) of Lemma \ref{lem:dimest}, $(x, -x_{\fk})\in\tilde{\fg}_{reg}$.  
Thus, both $x\in\fg_{reg}$ and $x_{\fk}\in\fk_{reg}$.  Hence by Equation (\ref{eq:regdiffs1}), 
\begin{equation}\label{eq:twoindep}
df_{n-1, 1}(x_{\fk})\wedge\dots\wedge df_{n-1, r_{n-1}}(x_{\fk})\neq 0\mbox{ and } df_{n,1}(x)\wedge\dots\wedge df_{n, r_{n}}(x)\neq 0.
\end{equation}
  Since $x\in(\fk_{\Delta}^{\perp})_{reg}$, $\fz_{\fk}(x_{\fk})\cap\fz_{\fg}(x)=0$ by Equation (\ref{eq:Kreg}).  
It now follows from (\ref{eq:twoindep}) and (\ref{eq:centralizer}) that 
$\omega_{\fg//K}(x)\neq 0$. 
\end{proof}
Theorem \ref{thm:Kostant} can be obtained as a special case of a more general result 
proven by Knop, \cite{Knopsmooth}.  We include our proof here because of its simplicity.  Theorem \ref{thm:Kostant} has an immediate corollary which 
is of interest in linear algebra.
\begin{cor}\label{c:linearalgebra}
Let $x\in\fg$ and suppose that $\fz_{\fk}(x_{\fk})\cap\fz_{\fg}(x)=0$.  
Then $x\in\fg_{reg}$ and $x_{\fk}\in\fk_{reg}$.
\end{cor}
\begin{proof}
This follows by Equation (\ref{eq:Kreg}) and Theorem \ref{thm:Kostant}.
\end{proof}

We observed in the proof of Lemma \ref{lem:dimest} that for $x\in\fg$, $\fz_{\fk}(x)=\fz_{\fk}(x_{\fk}) \cap \fz_{\fg}(x)$.
 Elements of $(\fk_{\Delta}^{\perp})_{reg}$  can be used
 to inductively construct strongly regular elements of $\fg$, so we give them a special name. 
 \begin{dfn}\label{dfn:nsreg}
 An element $x\in\fg$ such that $\fz_{\fk}(x)=0$ is said 
 to be $n$-strongly regular.  We denote the set of $n$-strongly regular elements by $\fg_{nsreg}$.  
 \end{dfn}

We can use Corollary \ref{c:linearalgebra} to simplify the characterization of
strongly regular elements originally given by Kostant and Wallach for the general linear case in Theorem 2.14 of \cite{KW1} 
and generalized by the first author to the orthogonal case in Proposition 2.11 of \cite{Col2}.

\begin{prop}\label{prop:fullsreg}
An element $x\in\fg$ is strongly regular if and only if 
$$
\fz_{\fg_{i}}(x_{i})\cap\fz_{\fg_{i+1}}(x_{i+1})=0 \mbox { for } i=2,\dots, n-1.
$$
\end{prop}
 \begin{proof}
 By Proposition 2.11 of \cite{Col2},
an element $x\in\fg$ is strongly regular if and only if the following two conditions hold:
\begin{equation*}
\begin{split}
&(1)\; x_{i}\in\fg_{i} \mbox{ are regular for all } i=2,\dots, n. \\
&(2)\; \fz_{\fg_{i}}(x_{i})\cap\fz_{\fg_{i+1}}(x_{i+1})=0 \mbox { for } i=2,\dots, n-1.
\end{split}
\end{equation*}
 It follows from Corollary \ref{c:linearalgebra} that if $x_{i+1}\in\fg_{i+1}$ 
satisfies $\fz_{\fg_{i}}(x_{i})\cap \fz_{\fg_{i+1}}(x_{i+1})=0$,  then $x_{i+1}\in\fg_{i+1}$ is regular.   
\end{proof}

\begin{rem}\label{r:generallin2}
We note that arguments given above also apply to the general linear case.  Indeed, we observed in Remark \ref{r:appendix} 
that if we let $G=GL(n)$, $K=GL(n-1)$, then the pair $(\tilde{G}, K_{\Delta})$ is spherical and satisfies Equation (\ref{eq:numerology}).  
Further, Lemma \ref{lem:dimest} also holds 
for this spherical pair, and we obtain Theorem \ref{thm:Kostant} by the same proof.  Corollary \ref{c:linearalgebra} follows, 
and we obtain a simplification of Kostant and Wallach's characterization of strongly regular elements in Theorem 2.14 of \cite{KW1}.

In \cite{CEeigen}, we defined the set of $n$-strongly regular elements 
 for $\fg=\fgl(n)$ to be the set of elements $x\in\fg$ for which $\omega_{\fg//K}(x)\neq 0$.  
 It follows from Theorem \ref{thm:Kostant} and Equation (\ref{eq:Kreg})
  that our definition in Definition \ref{dfn:nsreg} is consistent with the previous one and $\fg_{nsreg}
\cong (\fk_{\Delta}^{\perp})_{reg}.$

\end{rem}

The following technical result will be useful in the next section in our description 
of the geometry of generic fibres of the KW map (see Theorem \ref{thm:genericfibre}). 
\begin{lem}\label{l:centralizers}
For $x\in\fg_{nsreg}$, the group $Z_{K}(x)=Z_{K}(x_{\fk})\cap Z_{G}(x)=\{ e \}$ is the trivial group.
\end{lem}

%MARK: YOU WILL HAVE TO FIX THE NOTATION IN THIS PROOF AT SOME POINT.

\begin{proof}
Since $x\in \fg_{nsreg}$, $\mbox{Lie}(Z_{K}(x))=\fz_{\fk}(x)=0,$
so that $Z_K(x)$ is a finite group.  
Decompose $x=x_{\fk} + x_{\fp}$ with respect to the Cartan decomposition, so that $Z_K(x)=Z_K(x_{\fk}) \cap Z_K(x_{\fp})$.  Consider the Jordan decomposition of $x_{\fk}$ in $\fk$, $x_{\fk}=s+n$, with $s$ semisimple and $n$ nilpotent.
Consider the Levi subgroup $L:= Z_K(s)$ of $K$, let
$\fl=\mbox{Lie}(L)$, and let $Z$ be the centre of $L$.  We claim
that 
\begin{equation}\label{eq:Kcentralizer}
Z_K(x)=Z_{x_{\fp}}:=\{ z\in Z : \Ad(z)\cdot x_{\fp}=x_{\fp} \}.
\end{equation}
Indeed, if $z\in Z_{x_{\fp}}$, then since $s, n \in \fl$ and $z\in Z$, it follows
that $z\in Z_K(x_{\fk})$, and hence $z\in Z_K(x_{\fk}) \cap Z_K(x_{\fp})=Z_K(x)$.
Conversely, by standard properties
of the Jordan decomposition, $Z_K(x_{\fk})=Z_K(s) \cap Z_K(n) = Z_L(n)$.
Since $x\in \fg_{nsreg}$, $x_{\fk}\in\fk_{reg}$ by Corollary \ref{c:linearalgebra}, so $n$
is regular nilpotent in $\fl$, and hence 
$Z_L(n)=Z \cdot U$, where $U$ is a unipotent subgroup of $L$.
It follows that the finite subgroup $Z_K(x)$ of $Z_L(n)$ is in $Z$, and hence
that $Z_K(x)\subset Z_{x_{\fp}}$, and this establishes (\ref{eq:Kcentralizer}).
For later use, we let $Z^{0}$ denote the identity component of $Z$.

From the classification of Levi subgroups of $K=SO(n-1)$, it follows
that up to $K$-conjugacy, 
\begin{equation}\label{eq:Levis}
L=GL(s_1) \times \dots \times GL(s_d) \times
SO(r),
\end{equation}
 where $r\equiv n-1 \pmod{2}$, and $r\not= 2$, or 
 \begin{equation}\label{eq:Levis2}
 L=GL(m_1) \times \dots \times GL(m_d).
 \end{equation}
 If $n-1$ is even and $L$ is as in (\ref{eq:Levis}), then the centre $Z$ of $L$ is
$GL(1)^d \times <\epsilon>$, where $\epsilon$ is the negative of
the identity in $SO(r)$.  Otherwise,  $Z=Z^{0}=GL(1)^d$.    
 Recall also that the $K$-module $\fp \cong V$, where $V$ is the standard representation
of $K$.   It follows that the $Z^{0}$-weights of $\fp$ consist of
$\Gamma = \{ \pm \mu_i : i=1, \dots, d \}$, where $\mu_1, \dots, \mu_d$ is a basis of
the character group $X^*(Z^{0})$ of $Z^{0}$, along with the trivial character when $L$ is given by (\ref{eq:Levis}).    
    Let $x_{\fp}  = \displaystyle\bigoplus_{\lambda\in \Gamma} x_{\lambda}$
be the decomposition of $x_{\fp}$ into $Z^{0}$-weight vectors, and let
$\Gamma_0$ consist of the $\mu_i$ such $x_{\mu_i}$ or $x_{-\mu_i}$ is
nonzero.  Let $Z_{\Gamma_0}^{0}:=Z^{0} \cap Z_{x_{\fp}}$
and note that 
$$Z_{\Gamma_0}^{0}= \{ z\in Z^{0} : \mu_i(z)=1 
\mbox{ for all} \, \mu_i \in \Gamma_0 \}.$$    If $|\Gamma_0| < d$, then
$\dim(Z_{\Gamma_0}^{0}) \ge 1$, which contradicts the finiteness of $Z_K(x)$
 (\ref{eq:Kcentralizer}).  Hence $|\Gamma_0|=d$, so that $\Gamma_0$ is a basis of $X^*(Z^{0})$, and it follows that
 $Z_{\Gamma_0}^{0}=\{ e \}$.  When  $Z=Z^{0}$, we have 
 $$\{e\}=Z_{\Gamma_0}^{0}=Z_{x_{\fp}}=Z_{K}(x)$$
  by (\ref{eq:Kcentralizer}), and the lemma is proven.   If, on the other hand, $Z\not= Z^{0}$, then 
  by our remarks above, $n-1$ must be even and $L$ is given by (\ref{eq:Levis}).  
  In this case, we can decompose $\fp = \oplus_{i=1}^d \fp_{\pm \mu_i} \oplus \fp_0$,  where $\fp_{\mu_{i}}$ is the standard representation of $GL(s_{i})$, $\fp_{-\mu_{i}}$ is its contragradient, and $\fp_{0}$ is the standard representation of $SO(r)$.  Thus, $\epsilon$ acts
as the negative of the identity on $\fp_{0}$ while $Z^{0}$ acts trivially on $\fp_0$, and $\eps$ acts trivially on $\oplus_{i=1}^d \fp_{\pm \mu_i}$.
If $g\in Z_K(x) \setminus \{ e \}$, then since $Z_{\Gamma_{0}}^{0}=\{e\}$, we would have $g=z\cdot \epsilon$ with $z\in Z^{0}$.
 If this were the case, then $x_{\fp_0}=0$.   It follows that if 
$x_{\fk}=\oplus_{i=1}^d x_i \oplus y$, with $x_i \in \fgl(s_i)$ and
$y\in \fso(r)$, is the decomposition of $x_{\fk}$ in $\fl$,  then $y\in \fz_{\fk}(x_{\fk}) \cap \fz_{\fg} (x)=0$, so
$y=0.$  But then $x_{\fk}$ is not regular, which is a contradiction.
Thus, $Z_K(x)= \{ e \}$.
\end{proof}

\subsection{Generic Elements for GZ integrable systems} \label{ss:generic}
We use Proposition \ref{prop:fullsreg} to construct a new set $\fg_{\Theta}\subset\fg$ of strongly regular elements.    We then 
 use the set $\fg_{\Theta}$ to show that the GZ system is completely integrable on regular adjoint orbits of $\fg$. 

We begin by studying the relation between the spectra of $x$ and $x_{\fk}$.  
For $x\in\fg$, recall that $\sigma(x)$ denotes the spectrum of $x$ (see Notation \ref{nota:spectrum}).  We show the Zariski open subset:
$$
\fg(0)=\{x\in\fg: \sigma(x_{\fk})\cap\sigma(x)= \emptyset\}. 
$$
consists of $n$-strongly regular elements in the sense of Definition \ref{dfn:nsreg}.  We require a result from linear algebra.  
\begin{lem}\label{lem:linearalgebra} 
Let $V=V_{1}\oplus V_{2}$ be a direct sum decomposition of a 
finite dimensional complex vector space.  Let $X, Y \in \mbox{End}(V)$ with $Y\not= 0$.   Suppose that $Y(V_1) \subset V_1$, $Y(V_2)=0$, and $[X,Y]=0$.
Define $X_1:= \pi_{V_1} \circ X|_{V_{1}} \in \mbox{End}(V_{1})$ where $\pi_{V_1}$ is
projection off $V_2$.  Let $V(\lambda)$ be the generalized eigenspace of $X$ of eigenvalue $\lambda$ and let 
$V_{1}(\lambda)$ be the generalized eigenspace of $X_{1}$ of eigenvalue $\lambda$.  Then $Y(V(\lambda))\subset V_{1}(\lambda)$.
\end{lem}
\begin{proof}
Let $v\in V(\lambda)$.  Then there is $j\in\mathbb{Z}^{\geq 0}$ such that 
$(X-\lambda {\Id}_{V})^{j} v=0$ where ${\Id}_{V}$ is the identity operator on $V$.  
For $k\ge 0$, let $v_k=(X-\lambda {\Id}_{V})^{k} Yv$.  Since $[X,Y]=0$, then
$v_k \in \mbox{Im}(Y) \subset V_{1}$.   We show that 
$v_k=(X_1-\lambda {\Id}_{V_1})^{k} Yv$ by induction on $k$.  The case $k=0$
is clear, and note that if $u\in V_1$ and $(X-\lambda {\Id}_{V})u \in V_1$, then
$(X-\lambda {\Id}_{V})u = (X_1 - \lambda {\Id}_{V_1})u$.  The inductive step
follows easily from this observation since each $v_k \in V_{1}$.  By the choice
of $j$,   
\begin{equation*}
(X_1 - \lambda {\Id}_{V_1})^jYv = (X-\lambda {\Id}_{V})^{j}Yv = Y(X-\lambda {\Id}_{V})^{j}v = 0, 
\end{equation*}
and this establishes the lemma.
\end{proof}

%Using Proposition \ref{p:submatrix}, we can prove: %a fundamental result regarding the structure of 

\begin{thm}\label{thm:gzero}
Let $x\in\fg(0)$.  Then $x\in\fg_{nsreg}$.    
\end{thm}
\begin{proof}
Suppose $x\not\in\fg_{nsreg}$ so that $\fz_{\fk}(x)\neq 0$.  
We show that $\sigma(x)\cap\sigma(x_{\fk})\neq \emptyset$ by considering the types $ B, D$ separately.
Suppose that $\fg=\fso(2l)$.  We apply Lemma \ref{lem:linearalgebra} to 
the vector space $V=\C^{2l}=V_{1}\oplus V_{2}$, where 
$V_{1}=\mbox{span} \{ e_{\pm 1}, \dots, e_{\pm (l-1)}, e_{l}+e_{-l}\}$, and 
$V_{2}=\mbox{span} \{e_{l}-e_{-l}\}$ (see Section \ref{ss:soevenreal}).  We
take $X=x$ and $Y$ any nonzero element of $\fz_{\fk}(x)$, so
that $[X,Y]=0$.    The reader can check that
 $\fk$ annihilates $V_{2}$, so $Y(V_2)=0$.  The involution $\theta$ acts on 
$e_{\pm i}$ via $\theta(e_{\pm i})=e_{\pm i}$ for $i\neq l$ and $\theta(e_{l})=e_{-l}$ (see Example \ref{ex:roottypes}).   Since $\theta(Z\cdot \theta(v))=\theta(Z)\cdot v$ for any $Z\in \fg$ and $v\in V$, it follows that $\fk(V_1)\subset V_1$
and $\fp(V_1)\subset V_2$.   Hence, $Y(V_1) \subset V_1$, $x_{\fk}(V_1) \subset
V_1$ and $x_{\fp}(V_1)\subset V_2$.     Thus,  $x_{\fk}$ is the element
$X_1$ from Lemma \ref{lem:linearalgebra}.  Let $V=\bigoplus_{\mu\in\sigma(x)} V(\mu)$ be the decomposition of $V$ into 
generalized eigenspaces of $x$.  Since $Y\neq 0$, there exists $\mu\in\sigma(x)$ such that 
$Y(V(\mu))\neq 0$, so by Lemma \ref{lem:linearalgebra},  $V_{1}(\mu)\neq 0$.
We show by contradiction that if $Y(V(0))\not= 0$, then $\dim(V_1(0))>1$.  
 Indeed, otherwise, $V_1(0)=\C Yv$ for some nonzero $v\in V(0)$.
 Recall the nondegenerate bilinear form  $\beta(\cdot, \cdot)$  on $\C^{2l}$ given in Equation (\ref{eq:beta}) 
and note that $\beta|_{V_{1}\times V_{1}}$ is the bilinear form on $V_{1}\cong \C^{2l-1}$ defining $K=SO(2l-1)$ (see Section \ref{ss:orthoreal}).  
Since $\beta(v_1,v_2)=0$ unless $v_1\in V(\mu_1)$ and $v_2\in V(-\mu_1)$,
we may assume that $\beta(Yv, Yv)=1$.     
Since $Y\in\fz_{\fk}(x)\subset \fz_{\fk}(x_{\fk})$, $Y$ stabilizes the subspace
$V_1(0)$.  Since $Y\in \fk$, $\beta(Y^{2}v, Yv)=-\beta(Yv,Y^{2}v)=-\beta(Y^{2}v, Yv)$, and $\beta(Y^{2}v,Yv)=0$.  Then $\dim(V_1(0))=1$ implies that $Y^{2}v=0$,
so that $1=\beta(Yv, Yv)=-\beta(v, Y^{2}v)=0$.  This contradiction establishes
the claim.   Hence there is $\mu \in \C$ such that $V(\mu)$ and $V_1(\mu)$
are nonzero, and if $\mu=0$, $\dim(V_1(\mu)) > 1$. It follows that
$\mu \in \sigma(x) \cap \sigma(x_{\fk})$ (Notation \ref{nota:spectrum}), so $x\not\in \fg(0)$, completing
the proof for type $D$.

If $\fg=\fso(2l+1)$, we let $V=\C^{2l+1}$ and $V_1 = \mbox{span} \{ e_{\pm 1}, \dots, e_{\pm l} \}$,
$V_2=\mbox{span} \{ e_0 \}$ (see Section \ref{ss:sooddreal}). As above, we
take $Y$ to be any nonzero element of $\fz_{\fk}(x)$, and verify that $x$
and $Y$ satisfy the hypotheses of Lemma \ref{lem:linearalgebra} and
$x_{\fk}$ is the element $X_{1}$ from the lemma.  
Let $\beta(\cdot, \cdot)$ be the bilinear form from Equation (\ref{eq:beta}).
Then by Lemma \ref{lem:linearalgebra}, there is a generalized eigenspace $V(\mu)$ such that $Y(V(\mu))$ is a nonzero subspace of $V_1(\mu)$.  We claim that
if $\dim(V(0))=1$, then $Y(V(0))=0$.   Indeed, if $\dim(V(0))=1$, then
$V(0)$ is spanned by a nonzero vector $v$ and as above $Y(V(0))\subset V(0)$.
Then $\beta(Yv,v)=-\beta(v, Yv)=-\beta(Yv,v)$, so that $Yv=0$ since
$\beta$ is nondegenerate on $V(0)$.
Hence, there is $\mu$ such that $Y(V(\mu))$ is a nonzero subspace of $V_1(\mu)$
and either $\mu\not=0$ or $\mu=0$ and
$\dim(V(0))\ge 2$.   As above, we conclude that $x\not\in \fg(0)$.   
\end{proof}  

% COROLLARY 4.16 NEXT.

Let $c=(c_{r_{n-1}}, c_{r_{n}})\in \C^{r_{n-1}}\times\C^{r_{n}}$ and write $c_{r_{i}}=(c_{i, 1},\dots, c_{i, r_{i}})\in \C^{r_{i}}$ for $i=n-1,\, n$.
 Let $I_{n, c}$ be the ideal of $\C[\fg]$ generated by the functions $f_{i, j}-c_{i,j}$ for $i=n-1, \, n$ and $j=1,\dots, r_{i}$, and note that the zero set
of $I_{n,c}$ is $\Phi_n^{-1}(c)$.  
 \begin{cor}\label{c:isgenericallyrad}
Let $x\in \fg(0)$, and let $c=\Phi_n(x)$.
% Let $c=(c_{r_{n-1}}, c_{r_{n}})=\Phi_{n}(x)$ with $x\in\fg(0)$, so that $\Phi_n^{-1}(c)\subset\fg(0)$ by Remark \ref{r:spectrum}.  
 \begin{enumerate}
\item The ideal $I_{n,c}$ is the ideal of functions vanishing on $\Phi_n^{-1}(c)$, and the variety $\Phi_{n}^{-1}(c)$ is smooth.
\item The fibre $\Phi_{n}^{-1}(c)$ is a single closed $K$-orbit.
\end{enumerate}\end{cor}

\begin{proof}
As in Remark \ref{r:spectrum}, the fibre $\Phi_n^{-1}(c) \subset \fg(0)$.
 By Theorem \ref{thm:gzero} every element of the fibre $\Phi_{n}^{-1}(c)$ is $n$-strongly regular. Hence Theorem \ref{thm:Kostant} implies that the differentials 
 $\{df_{i,j}(x):\, i=n-1,\, n; \; j=1,\dots, r_{i}\}$ are independent for all $x\in \Phi_{n}^{-1}(c)$.   
By Theorem 18.15 (a) of \cite{Eis}, the ideal 
 $I_{n,c}$ is radical, so $I_{n,c}$ is the ideal of 
$\Phi^{-1}_{n}(c)$.  The smoothness
 of $\Phi^{-1}_{n}(c)$ now follows since the differentials of the generators of $I_{n,c}$ are independent at every point of $\Phi_{n}^{-1}(c)$.
For the second assertion, note first that 
$$\dim(K)=\dim(\fg)-\dim (\fg//K)=\dim( \Phi_{n}^{-1}(c)),$$
where the first equality follows from Equations (\ref{eq:flagvarieties}) and (\ref{eq:quotdim}), and the second equality follows from Proposition \ref{prop:flat} (3).  
%=\dim(\Phi_{n}^{-1}(c)).
By Lemma \ref{lem:dimest}, $\dim(\Ad(K)\cdot x)=\dim(K)$ for all 
$x\in \Phi_{n}^{-1}(c)$.  By Proposition \ref{prop:flat} (2),
each fibre $\Phi_{n}^{-1}(c)$ has a unique closed $K$-orbit, which
implies the assertion.
\end{proof} 

Consider the Zariski open subvariety of $\fg$:
\begin{equation}\label{eq:fgtheta}
\fg_{\Theta}:=\{x\in\fg:\, \sigma(x_{i})\cap \sigma(x_{i+1})=\emptyset \mbox{ for } i=2,\dots, n-1\}.
\end{equation}
\begin{prop}\label{prop:orthofgtheta}
The elements of $\fg_{\Theta}$ are strongly regular.
\end{prop}
\begin{proof}
If $x\in \fg_{\Theta}$, then $x_{i}\in\fg_{i}(0)$ for $i=3,\dots, n$.  By 
Theorem \ref{thm:gzero},  $\fz_{\fg_{i-1}}(x_{i-1})\cap\fz_{\fg_{i}}(x_{i})=0$.  
The result now follows from Proposition \ref{prop:fullsreg}. 
\end{proof}

\begin{cor}\label{c:thetafibre}
Let $x\in\fg_{\Theta}$.  Then $\Phi^{-1}(\Phi(x))=\Phi^{-1}(\Phi(x))_{sreg}$.
\end{cor}
\begin{proof}
By Remark \ref{r:spectrum},  $\Phi^{-1}(\Phi(x))\subset\fg_{\Theta}$ for $x\in\fg_{\Theta}$.  The result then follows from 
Proposition \ref{prop:orthofgtheta}.
\end{proof}

\begin{rem}\label{r:generallin3}
Let $\fg=\fgl(n)$.  In Theorem 5.15 of \cite{Col1}, the first author proved that
$\fg_{\Theta}\subset\fg_{sreg}$ for the analogously defined set $\fg_{\Theta}$. 
The methods of this section also prove that $\fg_{\Theta}\subset\fg_{sreg}$,
and our proof is significantly simpler than the proof in \cite{Col1}.
To prove Theorem \ref{thm:gzero} for $\fgl(n)$ we simply apply Lemma \ref{lem:linearalgebra} with 
$\fg=\fgl(n)$, $\fk=\fgl(n-1)$, $V=\C^{n}$, with $V_{1}=\mbox{span}\{e_{1},\dots, e_{n-1}\}$ and 
$V_{2}=\mbox{span}\{e_{n}\}$.  Proposition \ref{prop:orthofgtheta} follows, since the 
analogue of Proposition \ref{prop:fullsreg} also holds in this case as we observed in Remark \ref{r:generallin2}.  Thus, we can construct the strongly regular elements $\fg_{\Theta}$ of $\fg$ in both
orthogonal and general linear cases using the same framework.  

\end{rem}

We can now prove one of our main results.

\begin{thm}\label{thm:intsystem}
The restriction of the GZ functions $J_{GZ}$ to a regular adjoint orbit in $\fg$ forms a completely integrable system on the orbit. 
\end{thm}
\begin{proof}  
Let $x\in\fg_{reg}$ and let $\Ad(G)\cdot x$ be the adjoint orbit containing $x$.  
By Proposition \ref{prop:sregintegrable}, it suffices to show that Equation (\ref{eq:sregintersect}) holds.
Let $\chi:\fg\to \fg//G$ be the adjoint quotient.  
Since the Kostant-Wallach map $\Phi$ is surjective (Theorem \ref{thm:surj}), there exists $y\in \fg_{\Theta}$ 
such that $\chi(y)=\chi(x)$.  It follows from Proposition \ref{prop:orthofgtheta} that $y\in\fg_{sreg}$, whence 
$y\in\fg_{reg}$ by (\ref{eq:totreg}).  Therefore $y\in \chi^{-1}(\chi(x))\cap\fg_{reg}=\Ad(G)\cdot x$, and
$\Ad(G)\cdot x\cap\fg_{sreg}\neq\emptyset$.

\end{proof}

We end this section by describing the KW fibre $\Phi^{-1}(\Phi(x))$ for $x\in\fg_{\Theta}$.  
In particular, we generalize Corollary 5.18 of \cite{Col1} to the orthogonal setting using Proposition \ref{prop:orthofgtheta} 
and Corollary \ref{c:isgenericallyrad}.  Our argument below can also be used to give an easier proof 
of Corollary 5.18 of \cite{Col1} in the general linear case.  In the proof of the following theorem, we use repeatedly the easy fact that the projection 
$\pi_{i}: \fg\to\fg_{i}, \, \pi_{i}(x)=x_{i},$ is $\Ad(G_{i})$-equivariant.

\begin{thm} \label{thm:genericfibre}
Let $x\in\fg_{\Theta}$, and let $Z_{G_{i}}(x_{i})$ be the centralizer of $x_{i}\in\fg_{i}$ in $G_{i}$ viewed as a subgroup of $G$.
 Then the morphism:
\begin{equation}\label{eq:Psidefn}
\Psi:\displaystyle\prod_{i=2}^{n-1} Z_{G_{i}}(x_{i})\rightarrow \Phi^{-1}(\Phi(x))
\mbox{ given by }
\Psi(z_{2}, \dots, z_{n-1})=\Ad(z_{2})\dots\Ad(z_{n-1})\cdot x,
\end{equation}
$z_{i}\in Z_{G_{i}}(x_{i})$, is an isomorphism of non-singular algebraic varieties.
\end{thm}
\begin{proof}
We first note that the image of $\Psi$ is contained in 
the fibre $\Phi^{-1}(\Phi(x))$.  Indeed, observe that 
$\Ad(z_{n-1})\cdot x\in\Phi^{-1}(\Phi(x)),$ since $(\Ad(z_{n-1})\cdot x)_{i}=x_{i}$ for all $i\leq n-1$.  
In fact, for any $i=2,\dots, n-1$ we have:  
\begin{equation}\label{eq:icutoff} 
(\Ad(z_{i}\dots z_{n-1})\cdot x)_{i}=x_{i}.
\end{equation}
We prove Equation (\ref{eq:icutoff}) by downward induction on $i$, 
with the base case $i=n-1$ following from our discussion above.  
Suppose for any $j$ with $i<j\leq n-1$, we have 
$(\Ad(z_{j}\dots z_{n-1})\cdot x)_{j}=x_{j}$.  Now
\begin{equation}\label{eq:cutoffs1}
(\Ad(z_{i}\dots z_{n-1})\cdot x)_{i}=[(\Ad(z_{i}\dots z_{n-1})\cdot x)_{i+1}]_{i}.
\end{equation}
Since $z_{i}\in G_{i}\subset G_{i+1}$, we have 
$$
(\Ad(z_{i}\dots z_{n-1})\cdot x)_{i+1}=\Ad(z_{i})\cdot( \Ad(z_{i+1}\dots z_{n-1})\cdot x)_{i+1}=\Ad(z_{i})\cdot x_{i+1}
$$
by the induction hypothesis.  
But then it follows from (\ref{eq:cutoffs1}) that 
$$
(\Ad(z_{i}\dots z_{n-1})\cdot x)_{i}=(\Ad(z_{i})\cdot  x_{i+1})_{i}=\Ad(z_{i})\cdot x_{i}=x_{i},
$$ 
since $z_{i}\in Z_{G_{i}}(x_{i})$, yielding (\ref{eq:icutoff}). 

Using Equation (\ref{eq:icutoff}), we now show that for any $i=2,\dots, n$, 
\begin{equation}\label{eq:adjoint}
\chi_{i}((\Ad(z_{2}\dots z_{n-1})\cdot x)_{i})=\chi_{i}(x_{i}),
\end{equation}
 where $\chi_{i}:\fg_{i}\to \fg_{i}//G_{i}$ is the adjoint quotient for $\fg_{i}$, and thus 
 $\Ad(z_{2}\dots z_{n-1})\cdot x\in \Phi^{-1}(\Phi(x))$ by the definition of the KW map in (\ref{eq:KWmap}).  
 We first note that (\ref{eq:adjoint}) is easily seen to be true for $i=n$, since $z_{2}\dots z_{n-1}\in G_{n-1}\subset G$.  
 For $i<n$, 
 \begin{equation}\label{eq:something}
 (\Ad(z_{2}\dots z_{i-1} z_{i}\dots z_{n-1})\cdot x)_{i}=\Ad(z_{2}\dots z_{i-1})\cdot (\Ad(z_{i}\dots z_{n-1})\cdot x)_{i}=\Ad(z_{2}\dots z_{i-1})\cdot x_{i}
\end{equation} 
by Equation (\ref{eq:icutoff}). 
Since $z_{2}\dots z_{i-1}\in G_{i-1}\subset G_{i}$, Equation (\ref{eq:something}) implies
$$
\chi_{i}((\Ad(z_{2}\dots z_{n-1})\cdot x)_{i})=\chi_{i}(\Ad(z_{2}\dots z_{i-1})\cdot x_{i})=\chi_{i}(x_{i}), 
$$
yielding (\ref{eq:adjoint}) in this case.

We now claim $\Psi$ is surjective.  Suppose $y\in\Phi^{-1}(\Phi(x))$.  Let $\Phi(x)=(c_{2},\dots, c_{n})\in \C^{r_{2}}\times\dots\times \C^{r_{n}}$.  Since $\chi_2$
is bijective, we see that $x_{2}=y_{2}$.  Further, $y_{3}, \, x_{3}\in \Phi_{3}^{-1}(c_{2}, c_{3})$, where $\Phi_{3}$ is the partial KW map for $\fg_{3}$.  
Since $x\in \fg_{\Theta}$, it follows from definitions that $\Phi_{3}^{-1}(c_{2}, c_{3})\subset \fg_{3}(0)$.  By part (2) of Corollary \ref{c:isgenericallyrad}, $y_{3}$, $x_{3}$ are $G_{2}=Z_{G_{2}}(x_{2})$-conjugate.  Let $z_{2}\in Z_{G_{2}}(x_{2})$ be such that $\Ad(z_{2})\cdot y_{3}=x_{3}$.  Let $y^{\prime}:=\Ad(z_{2})\cdot y$, so that
  $y^{\prime}_{3}=x_{3}$. Now observe that $y^{\prime}_{4},\, x_{4}\in \Phi_{4}^{-1}(c_{3}, c_{4})\subset \fg_{4}(0)$.  
  Again, by part (2) of Corollary \ref{c:isgenericallyrad}, $y_{4}^{\prime}$, $x_{4}$ are $G_{3}$-conjugate.  
  But since $y^{\prime}_{3}=x_{3}$, we see that $y_{4}^{\prime}$, $x_{4}$ are  $Z_{G_{3}}(x_{3})$-conjugate.  
  Let $z_{3}\in Z_{G_{3}}(x_{3})$ be such that $\Ad(z_{3})\cdot  y^{\prime}_{4}=x_{4}$.  
  Let $y^{\prime\prime}:=\Ad(z_{3} z_{2})\cdot y$.  Then we claim $y^{\prime\prime}_{4}=x_{4}$.  Indeed, 
  $$
  y_{4}^{\prime\prime}=(\Ad(z_{3}z_{2})\cdot y)_{4}=(\Ad(z_{3})\cdot y^{\prime})_{4}=\Ad(z_{3})\cdot y_{4}^{\prime}=x_{4}.
  $$
Continuing, in this fashion we can find $z_{2},\dots, z_{n-1}\in Z_{G_{2}}(x_{2}),\, \dots, Z_{G_{n-1}}(x_{n-1})$ 
respectively such that $x=\Ad(z_{n-1}\dots z_{2})\cdot y$.  It follows immediately that 
$\Psi(z_{2}^{-1},\dots, z_{n-1}^{-1})=y$  

We now show that $\Psi$ is injective.  Our main tool will be Lemma \ref{l:centralizers}.
Suppose we have $z_{i},\, \widetilde{z_{i}}\in Z_{G_{i}}(x_{i})$, for $i=2,\dots, n-1$ such that 
\begin{equation}\label{eq:firsteq}
\Ad(z_{2}\dots z_{n-1})\cdot x=\Ad(\widetilde{z_{2}}\dots \widetilde{z_{n-1}})\cdot x.
\end{equation}
Then $(\Ad(z_{2}\dots z_{n-1})\cdot x)_{3}=(\Ad(\widetilde{z_{2}}\dots\widetilde{z_{n-1}})\cdot x)_{3}$, 
which is equivalent to $\Ad(z_{2})\cdot x_{3}=\Ad(\widetilde{z_{2}})\cdot x_{3}$ by (\ref{eq:icutoff}).  
It follows that $z_{2}^{-1}\widetilde{z_{2}}\in Z_{G_{2}}(x_{2})\cap Z_{G_{3}}(x_{3})$. Since $x\in\fg_{\Theta}$, $x\in\fg_{sreg}$ by 
Proposition \ref{prop:orthofgtheta}, so that $x_{3}\in(\fg_{3})_{3-sreg}$ by Proposition \ref{prop:fullsreg}.  Lemma \ref{l:centralizers} now implies that 
$z_{2}=\widetilde{z_{2}}$.  
By induction, we may assume that $z_{3}=\widetilde{z_{3}}, \, z_{4}=\widetilde{z_{4}}, \dots, z_{i-1}=\widetilde{z_{i-1}}$, so that 
Equation (\ref{eq:firsteq}) becomes
$$
\Ad(z_{i}\dots z_{n-1})\cdot x=\Ad(\widetilde{z_{i}}\dots \widetilde{z_{n-1}})\cdot x.
$$
Using (\ref{eq:icutoff}) again, we have $\Ad(z_{i})\cdot x_{i+1}=\Ad(\widetilde{z_{i}})\cdot x_{i+1}$, yielding
$z_{i}^{-1}\widetilde{z_{i}}\in Z_{G_{i}}(x_{i})\cap Z_{G_{i+1}}(x_{i+1})$.  Lemma \ref{l:centralizers} again 
gives that $z_{i}=\widetilde{z_{i}}$.  By induction $\Psi$ is injective.

Finally, we show that $\Phi^{-1}(\Phi(x))$ is a nonsingular variety.  
 Let 
$\Phi(x)=(c_{i,j})_{i=2,\dots, n}^{j=1,\dots ,r_{i}} \in \C^{r_{2}}\times \dots\times \C^{r_{n}}.$  
Let $I_{c}\subset\C[\fg]$ be the ideal generated by the functions $g_{i,j}=f_{i,j}-c_{i,j}$ for 
$i=2,\dots, n$ and $j=1,\dots, r_{i}$.  
By Corollary \ref{c:thetafibre}, 
$\Phi^{-1}(\Phi(x))\subset\fg_{sreg}$.  It follows 
from the definition of strong regularity in Notation-Definition \ref{dfnote:sreg} that the differentials of the functions $g_{i,j}$ are independent at any 
point $y\in \Phi^{-1}(\Phi(x))$.  Again using Theorem 18.15 of \cite{Eis}, we
see that $y$ is a smooth point of $\Phi^{-1}(\Phi(x))$.
Thus, $\Psi: Z_{G_{2}}(x_{2}) \times \dots \times Z_{G_{n-1}}(x_{n-1})\to \Phi^{-1}(\Phi(x))$ is 
a bijective morphism of non-singular varieties.  It follows from Zariski's main theorem that $\Psi$ is a isomorphism.
\end{proof}

\subsection{Generic orbits of the GZ group $A$ on $\fg$}
The GZ integrable system in Equation (\ref{eq:GZfuns}) integrates to a global, holomorphic action of $A:=\C^{d}$ on 
$\fg$ where $$d=\frac{\dim \fg-r_{n}}{2}$$ is half the dimension of a regular $\Ad(G)$-orbit on $\fg$.  
To see this, one considers the Lie algebra of Hamiltonian GZ vector fields on $\fg$:
\begin{equation}\label{eq:GZvecfields}
\fa_{GZ}=\mbox{span}\{\xi_{f}:\, f\in J_{GZ}\}.
\end{equation}
For $f\in J_{GZ}$, the Hamiltonian vector field $\xi_{f}$ is complete and integrates to a global 
action of $\C$ on $\fg$ (see Theorem 2.4,\cite{Col2}).  Since the functions $J_{GZ}$ Poisson commute, the Lie algebra 
$\fa_{GZ}$ is abelian, and therefore the global flows of the vector fields $\xi_{f}$ simultaneously integrate to 
give a holomorphic action of $A$ on $\fg$ (see Section 2.3, \cite{Col2} for details).  We refer to the group $A$ 
as the Gelfand-Zeitlin group (or GZ) group (see \cite{Col1},\cite{CEexp}).  Since the Lie algebra 
$\fa_{GZ}$ is commutative, the GZ functions $f_{i,j}\in J_{GZ}$, $i=2,\dots, n,\,  j=1,\dots, r_{i}$ are invariant under one 
another's Hamiltonian flows, whence the action of $A$ preserves the fibres of the KW map $\Phi$.  In fact, Theorem 
2.14 of \cite{Col2} implies that the irreducible components of the strongly regular fibre $\Phi^{-1}(\Phi(x))_{sreg}$ are precisely the
orbits of $A$ on the fibre $\Phi^{-1}(\Phi(x))$.  Theorem \ref{thm:genericfibre} can be now 
be used to describe the action of the GZ group $A$ on the fibres $\Phi^{-1}(\Phi(x))$ for $x\in\fg_{\Theta}$.
\begin{thm}\label{thm:GZgroup}
Let $x\in\fg_{\Theta}$, and let $Z_{G_{i}}(x_{i})^{0}$ be the 
identity component of the group $Z_{G_{i}}(x_{i})$.  Let 
$$\mathcal{Z}:=\displaystyle\prod_{i=2}^{n-1} Z_{G_{i}}(x_{i})^{0}.$$
Then the $A$-orbits on $\Phi^{-1}(\Phi(x))$ coincide with orbits of a free, algebraic action of the 
connected, abelian algebraic group $\mathcal{Z}$ on $\Phi^{-1}(\Phi(x))$.  
In particular, the Lie algebra of GZ vector fields in (\ref{eq:GZvecfields}) is  algebraically integrable 
on the fibre $\Phi^{-1}(\Phi(x))$.  
\end{thm}

\begin{proof}
By Theorem \ref{thm:genericfibre}, the morphism 
$$
\Psi: \displaystyle\prod_{i=2}^{n-1} Z_{G_{i}}(x_{i})\to \Phi^{-1}(\Phi(x)) \mbox{ given by } \Psi(z_{2},\dots, z_{n})=\Ad(z_{2})\cdots \Ad(z_{n-1})\cdot x.
$$
is an isomorphism of varieties (cf. Equation (\ref{eq:Psidefn})).  
Therefore, we can use $\Psi$ to define an action of $\mathcal{Z}$ on $\Phi^{-1}(\Phi(x))$ via 
\begin{equation}\label{eq:actdefn}
(z_{2},\dots, z_{n})\cdot y:=\Psi((z_{2},\dots, z_{n})\circ \Psi^{-1}(y)), 
\end{equation}
where $\circ$ in the above equation represents multiplication in the group 
$\prod_{i=2}^{n-1} Z_{G_{i}}(x_{i})$.  The action of $\mathcal{Z}$ in (\ref{eq:actdefn}) is clearly 
algebraic and free and its orbits are the irreducible components of $\Phi^{-1}(\Phi(x))$, since $\Psi$ is an isomorphism of varieties.  
Since $\Phi^{-1}(\Phi(x))=\Phi^{-1}(\Phi(x))_{sreg}$ by Corollary \ref{c:thetafibre}, the orbits of $\mathcal{Z}$ are then precisely the 
orbits of the GZ group $A$ on $\Phi^{-1}(\Phi(x))$ by our remarks above.
\end{proof}

Using Theorem \ref{thm:genericfibre}, we can count the number of irreducible components (and hence the number of $A$-orbits)
in $\Phi^{-1}(\Phi(x))$ for $x\in\fg_{\Theta}$. 
\begin{cor}\label{c:counting}
 Let $\fg=\fso(n)$ with $n> 3$, and let $x\in\fg_{\Theta}$.  
 Let $m$ be the number of indices $i$, $i=4,\dots, n-1$, satisfying the following two conditions:
  \begin{enumerate}
 \item $i$ is even.
 \item zero occurs as an eigenvalue of $x_{i}$ with multiplicity at least 4.
 \end{enumerate}
 Then $m\in \{0,\dots, r_{n-1} - 1\}$ and the number of irreducible components of
 $\Phi^{-1}(\Phi(x))$ is $2^{m}$.
 
Further, for any $m\in\{0,\dots, r_{n-1} - 1\}$, there exists an $x\in\fg_{\Theta}$ with the property 
that $\Phi^{-1}(\Phi(x))$ contains $2^{m}$ irreducible components. 
 \end{cor}
\begin{proof}
It follows immediately from the isomorphism in Theorem \ref{thm:genericfibre} that the number of 
irreducible components of $\Phi^{-1}(\Phi(x))$ for $x\in\fg_{\Theta}$ coincides with the number of 
irreducible (i.e. connected) components of the product of algebraic groups $Z_{G_{2}}(x_{2})\times\dots\times Z_{G_{n-1}}(x_{n-1})$.  
We compute the number of components of $Z_{G_{i}}(x_{i})$ by considering the Jordan decomposition of $x_{i}\in\fg_{i}$, $x_{i}=s_{i}+n_{i}$ 
with $s_{i}$ semisimple and $n_{i}$ nilpotent.  Then $Z_{G_{i}}(x_{i})=Z_{L_{i}}(n_{i})$, where $L_{i}:=Z_{G_{i}}(s_{i})$ is a Levi subgroup of $G_{i}$.  
Since $x\in\fg_{\Theta}\subset \fg_{sreg}$ by Proposition \ref{prop:orthofgtheta}, $x_{i}$ is regular for all $i=2,\dots, n$ by (\ref{eq:totreg}), so that 
$n_{i}\in \fl_{i}=\mbox{Lie}(L_{i})$ is regular nilpotent.  Thus, $Z_{L_{i}}(n_{i})\cong Z_{i}\times U_{i}$, where $Z_{i}$ is the centre of $L_{i}$ and $U_{i}$ is a unipotent subgroup of $L_{i}$.  Suppose $i>2$ is even.  As we observed in the proof of Lemma \ref{l:centralizers}, up to $G_{i}$-conjugacy the Levi subgroup $L_{i}$ is either a product of subgroups $GL(m_{j})$ in which case $Z_{i}$ is connected, or $L$ is a product of $GL(s_{j})$ and precisely one factor of the form $SO(2k)$ with $k>1$, in which case $Z_{i}$ has exactly two components (see Equations (\ref{eq:Levis}) and (\ref{eq:Levis2})).  
The latter case occurs if and only if zero occurs as an eigenvalue of $s_{i}$ of 
multiplicity at least $4$.  On the other hand, if $i$ is odd, a similar argument shows that $Z_{G_{i}}(x_{i})$ is always a connected, algebraic group.  
Lastly, if $i=2$, the group $Z_{G_{2}}(x_{2})=G_{2}$ is connected. 

The upper bound on the value of $m$ follows from an easy calculation that there are precisely $r_{n-1}-1$ subalgebras of type $D$ in the chain
$\fg_{3}\subset \dots\subset\fg_{n-1}$. 

To prove the second statement of the theorem, suppose we are given $m\in\{0,\dots, r_{n-1}-1\}$.  Then we can find $m$ subalgebras 
$\fg_{i_{1}}\subset\fg_{i_{2}}\subset\dots\subset\fg_{i_{m}}$ of type $D$ in the chain $\fg_{3}\subset\dots\subset\fg_{n-1}$.  It follows 
from the surjectivity of the KW map (Theorem 3.6) that we can find $x\in\fg_{\Theta}$ with $x_{i_{j}}$ regular nilpotent in $\fg_{i_{j}}$ for all $j=1,\dots, m$.  
It follows that $\Phi^{-1}(\Phi(x))$ contains $2^{m}$ irreducible components from the first statement of the theorem.

\end{proof}

\begin{rem}\label{r:fgomega}
Theorems \ref{thm:genericfibre} and \ref{thm:GZgroup} improve on results
from \cite{Col2} and \cite{Col1}.  In \cite{Col2}, the first author
proved Theorem \ref{thm:genericfibre} for the elements $x$ of $\fg$ such
that $x_i$ is regular semisimple for $i=2, \dots, n$ and $x_i$ and $x_{i+1}$
have disjoint spectra for $i=2, \dots, n-1$.  In this case, the fibres
$\Phi^{-1}(\Phi(x))$ are isomorphic to a product of tori.
   Our result here is more
general and has a simpler proof.  In \cite{Col1}, the first author proved
the analogue of Theorem \ref{thm:GZgroup} for $\fg = \fgl(n)$.  Our
proof here carries over to $\fgl(n)$ with minor modifications, and is 
simpler than the proof in \cite{Col1}.
\end{rem}

\section{The orthogonal KW nilfibre and strongly regular elements}\label{s:KWnilfibre}
In this section, we show that $\Phi^{-1}(0)_{sreg}=\emptyset$ (see Proposition \ref{p:nosreg}).
This stands in stark contrast to the situation in the general linear setting where the rich structure of $\Phi^{-1}(0)_{sreg}$ is studied 
extensively in \cite{CEKorbs}.  As with the generic fibres of $\Phi$ studied in the previous sections, our study 
of $\Phi^{-1}(0)$ begins with studying the nilfibre of the partial KW map $\Phi_{n}^{-1}(0)$.

\subsection{Structure of Partial KW nilfibre}\label{ss:partialKWnilfibre}

Our goal in this section is to describe the structure of 
the nilfibre $\Phi_{n}^{-1}(0)$  of the partial KW map  (see Notation \ref{nota:partialnil}).  
This will be achieved by degenerating a generic fibre $\Phi_{n}^{-1}(\Phi_{n}(x))$
 to $\Phi_{n}^{-1}(0)$ together with the Luna slice theorem.

We recall the basic ingredients of Luna's slice theorem.
Let $M$ be a complex reductive algebraic group acting linearly 
on a finite dimensional vector space $V$.  Let $M\cdot v$ be a closed $M$-orbit
in $V$ and note that the stabilizer $M_{v}$ is reductive.  Hence,
we may find a representation $\fs$ of $M_{v}$ such that
$V \cong T_{v}(M\cdot v)\oplus \fs$ as $M_{v}$-modules.   The 
representation $\fs$ is called the slice representation at $v$.
Let $q:\fs\to \fs// M_{v}$ be the GIT quotient for the slice representation, and let $\N_{\fs}:=q^{-1}(0) = \{ y \in \fs : 0 \in \overline{M\cdot y}\}$ be the nullcone 
of the slice representation. Consider the GIT quotient 
$\pi: V\to V//M$ for the $M$-action on $V$.   The Luna slice theorem
asserts that the $M$-equivariant morphism 
\begin{equation}\label{eq:Lunaiso}
M\times_{M_{v}}(v+\N_{\fs})\stackrel{\cong}{\to} \pi^{-1}(\pi(v))\mbox{ given by } (g, v+n)\mapsto g\cdot (v+n)
\end{equation}
is an isomorphism, where $g \in M$ and $n\in\N_{\fs}$ (Theorem 6.6 of \cite{VP},  \cite{Luna}).  In particular, 
the $M$-orbits on $\pi^{-1}(\pi(v))$ are in one-to-one correspondence with the 
$M_{v}$-orbits in $\N_{\fs}$. 

To apply this construction to the $K=SO(n-1)$-action on $\fg=\fso(n)$, we recall
some facts about the involution $\theta$ in Section \ref{ss:symmetricreal}.  
Suppose that $\fb\in\B$ is $\theta$-stable, i.e., $\theta(\fb)=\fb$, and let $\ft\subset\fb$ be a $\theta$-stable Cartan subalgebra of $\fg$.  Then the corresponding Borel subgroup $B\subset G$ with $\mbox{Lie}(B)=\fb$ is also $\theta$-stable and contains the $\theta$-stable Cartan subgroup $T\subset G$ with $\mbox{Lie}(T)=\ft$.  It follows from Lemma 5.1 of \cite{Rich} 
that $B\cap K$ is a Borel subgroup of $K$ and $T\cap K$ is a Cartan subgroup of $K$.  
Hence, $\fb\cap\fk$ is a Borel subalgebra of $\fk$ with nilradical $\fn\cap\fk=[\fb\cap\fk,\fb\cap\fk]$, and $\fb\cap\fk$ contains the Cartan subalgebra $\ft\cap\fk$ of $\fk$. 

Recall the notion of a standard Borel subalgebra in Definition \ref{dfn:std}.  
We will make use of the following notation throughout this section. 

\begin{nota}\label{nota:thetastable}
Let $\B^{\theta}\subset\B$ denote the set of $\theta$-stable Borel subalgebras of 
$\fg$.  We denote by $\B^{\theta}_{std}\subset \B^{\theta}$ the subset of standard $\theta$-stable Borel subalgebras 
of $\fg$.  Similarly, we let 
\begin{equation*}
\mathcal{N}^{\theta}:=\{\fn=[\fb,\fb]\mbox{ with } \fb\in\B^{\theta}\}, 
\end{equation*}
and 
$$\mathcal{N}^{\theta}_{std}:=\{ \fn=[\fb,\fb]\mbox{ with } \fb\in\B^{\theta}_{std}\}$$
be the collections of the $\theta$-stable nilradicals and $\theta$-stable, standard nilradicals respectively.  
Note that the standard Borel subalgebra of upper triangular matrices $\fb_{+}$ in $\fg$  
belongs to the set $\B^{\theta}_{std}$.   
\end{nota}

 Recall that $\fh\subset\fg$ denotes the standard Cartan
subalgebra of diagonal matrices and that $\fh$ is preserved by $\theta$.  Since $\fb_{+}\in \B^{\theta}_{std}$, it follows
that $\fh\cap\fk$ is a Cartan subalgebra of $\fk$ with corresponding Cartan subgroup $H\cap K$ of $K$.  
Let $(\fh\cap\fk)_{reg}$ denote the regular semisimple elements of $\fk$ in $\fh \cap \fk$.   Then
for $x\in (\fh \cap \fk)_{reg}$, the orbit $\Ad(K)\cdot x$ is closed in $\fk$,
and thus in $\fg$, and the isotropy group of $x$ is the Cartan subgroup $H\cap K$ of $K$.  Let $\Phi_{\fk}$ be the
$\fh \cap \fk$-roots in $\fk$.   Then we may identify 
$T_{x}(\Ad(K)\cdot x)=[\fk,x] = \oplus_{\alpha \in \Phi_{\fk}} \fk_{\alpha}$.
Let $\fs=\fh + \fg^{-\theta}$, so that $\fs$ is
 a $H\cap K$-slice to $\Ad(K)\cdot x$ in $\fg$.   Note that the $H\cap K$-nullcone
$\N_{\fs}=\N_{\fg^{-\theta}}$, since $H\cap K$ acts trivially on $\fh$.

We compute $\N_{\fg^{-\theta}}$ by considering type $B$ and type $D$ separately.
Let $\fg=\fso(2l+1)$ and recall from 
 Example \ref{ex:roottypes} that $H\cap K = H$ and 
$\fg^{-\theta}=\oplus_{i=1}^{l} \fg_{\pm \eps_{i}}$.     An element of $H\cap K$
is $h=\mbox{diag}[h_{1},\dots, h_{l} ,1, h_{l}^{-1}, \dots, h_{1}^{-1}]$ with $h_{i}\in\C^{\times}$, and note that if $e_{\pm\eps_{i}}$ a root vector of $\fg_{\pm\eps_{i}}$, then $\Ad(h)\cdot e_{\pm\eps_{i}} = h_i^{\pm 1} e_{\pm \eps_{i}}$.   Hence, if 
$x=\sum_{i=1}^{l} \lambda_{i} e_{\eps_{i}} + \mu_{i} e_{-\eps_{i}}$ with $\lambda_i, \mu_i \in \C$, then
\begin{equation*}
\Ad(h)\cdot x=\displaystyle \sum_{i=1}^{l} \lambda_{i} h_{i} e_{\eps_{i}}+ \mu_{i} h_{i}^{-1} e_{-\eps_{i}}. 
\end{equation*}
It follows that
\begin{equation*}
0\in \overline{\Ad(H)\cdot x} \mbox{ if and only if } \mu_{i}\lambda_{i}=0 \mbox{ for all } i=1,\dots, l. 
\end{equation*}
Now let $\fg=\fso(2l)$, and note that since $\theta(\eps_{l})=-\eps_{l}$
 by Example \ref{ex:roottypes},  
$$
\fg^{-\theta}= \fh^{-\theta} \oplus \displaystyle \bigoplus_{i=1}^{l-1} (\fg_{\eps_{i}-\eps_{l}}\oplus \fg_{\eps_{i}+ \eps_{l}})^{-\theta}\oplus \displaystyle\bigoplus_{i=1}^{l-1}(\fg_{-(\eps_{i}-\eps_{l})}\oplus \fg_{-(\eps_{i}+\eps_{l})})^{-\theta}.
$$
Let $f_{\pm i}$ be a root vector of $\fg_{\pm(\eps_{i}-\eps_{l})}$ for $i=1,\dots, l-1$, 
so that $\fg^{-\theta} = \fh^{-\theta} + \sum_{i=1}^{l-1} \C (f_{\pm i} -\theta(f_{\pm i}))$.    Elements of $H\cap K$ are of the form $h=\mbox{diag}[h_{1},\dots, h_{l-1}, 1,1, h_{l-1}^{-1},\dots, h_{1}^{-1}]\subset H$ with $h_{i}\in \C^{\times}.$
Then $\Ad(h)\cdot (f_{\pm i} - \theta(f_{\pm i}))=h_i^{\pm 1} (f_{\pm i} - \theta(f_{\pm i}))$.  
Thus, if $$x=x_{\fh^{-\theta}}+\sum_{i=1}^{l-1} \lambda_{i}(f_{i}-\theta(f_{i}))+ \mu_{i}(f_{-i}-\theta(f_{-i})),$$ with $x_{\fh^{-\theta}}\in\fh^{-\theta}$, then 
\begin{equation}\label{eq:typeDact}
\Ad(h)\cdot x=x_{\fh^{-\theta}}+\displaystyle\sum h_{i}\lambda_{i}(f_{i}-\theta(f_{i}))+h_{i}^{-1} \mu_{i}(f_{-i}-\theta(f_{-i})),
\end{equation}
since $H\cap K$ acts trivially on $\fh^{-\theta}$.  
It follows that
  \begin{equation*}
0\in \overline{\Ad(H\cap K)\cdot x} \mbox{ if and only if } \mu_{i}\lambda_{i}=0 \mbox{ for all } i=1,\dots, l-1, \mbox{ and } x_{\fh^{-\theta}}=0.
\end{equation*}
To describe the irreducible components of $\N_{\fg^{-\theta}}$, introduce symbols
$U$ and $L$ for upper and lower.  For a $r_{n-1}$-tuple of symbols 
$(i_{1},\dots, i_{r_{n-1}})$ where $i_{j}=U$ or $i_{j}=L$ for $j=1,\dots, r_{n-1}$, define
\begin{equation}\label{eq:Cdefn}
\mathcal{C}_{i_{1},\dots, i_{r_{n-1}}}:=\{x\in \fg^{-\theta}: \mu_{j}=0\mbox { if } i_{j}=U,\, \lambda_{j}=0 \mbox{ if } i_{j}=L,\mbox{ and } x_{\fh^{-\theta}}=0 \mbox{ if }\fg=\fso(2l)\}.
\end{equation}
We have proved:

\begin{lem}\label{lem:nullconeirredcomponets}
The irreducible components of $\N_{\fg^{-\theta}}$ are the $2^{r_{n-1}}$ varieties
$\mathcal{C}_{i_{1},\dots, i_{r_{n-1}}} \cong \C^{r_{n-1}}$.   They are indexed by
the choices of $r_{n-1}$-tuples $(i_{1}, \dots, i_{r_{n-1}})$, with $i_{j}=U$ or $i_{j}=L$ for $j=1,\dots, r_{n-1}$. 
\end{lem}

We now give a Lie-theoretic description of the components
$\mathcal{C}_{i_{1},\dots, i_{r_{n-1}}}$.

\begin{lem}\label{l:dim}
Let $\fb\in \B^{\theta}$ with
nilradical $\fn$.  Then 
$\dim \fn^{-\theta}=r_{n-1}=\mbox{\emph{rank}}(\fk)$.
\end{lem}
\begin{proof}
Since $\fb\in\B^{\theta}$, $\fn\in\mathcal{N}^{\theta}$, and therefore
\begin{equation}\label{eq:ndim}
\dim\fn=\dim \fn^{-\theta}+\dim(\fn\cap\fk).
\end{equation}
By our discussion above, $\fn \cap \fk$ is the nilradical of the Borel subalgebra $\fb\cap\fk$ of $\fk$, so
by Equation (\ref{eq:sumri}),
$$\dim \fn=\frac{1}{2}(\dim\fg-r_{n})=\sum_{i=2}^{n-1} r_{i}, \ \mbox{and} 
\dim(\fn\cap\fk)=\sum_{i=2}^{n-2} r_{i}.$$    
Thus, Equation (\ref{eq:ndim}) becomes 
$$
\dim \fn^{-\theta}=\dim \fn-\dim \fn\cap\fk=\displaystyle \sum_{i=2}^{n-1} r_{i}-\displaystyle\sum_{i=2}^{n-2} r_{i}=r_{n-1}.
$$
\end{proof}

\begin{prop}\label{p:nullconeslice}
Let $\N_{\fg^{-\theta}}$ be the nullcone for the $H\cap K$-action on $\fg^{-\theta}$, and 
let $\mathcal{C}_{i_{1},\dots, i_{r_{n-1}}}$ be an irreducible component of 
$\N_{\fg^{-\theta}}$.  Then there exists $\fn\in\mathcal{N}^{\theta}_{std}$ such that 
$\mathcal{C}_{i_{1},\dots, i_{r_{n-1}}}=\fn^{-\theta}.$  

Conversely, if $\fn\in\mathcal{N}_{std}^{\theta}$, then $\fn^{-\theta}$ is an irreducible component of $\N_{\fg^{-\theta}}$ and therefore
$\fn^{-\theta}=\mathcal{C}_{i_{1},\dots, i_{r_{n-1}}}$ for some choice of indices $i_{j}=U,\, L$ as in 
(\ref{eq:Cdefn}).
\end{prop}

\begin{proof}
We prove the proposition when $\fg=\fso(2l)$ is type $D$.  
The case where $\fg$ is type $B$ is similar and left to the reader.  
Suppose $\mathcal{C}_{i_{1},\dots, i_{l-1}}\subset \N_{\fg^{-\theta}}$ is an irreducible component.  
Then there exists an $x\in (\fh\cap\fk)_{reg}$ with $x\in\fh_{reg}$ such that 
$\ad(x)$ acts on $\mathcal{C}_{i_{1},\dots, i_{l-1}}$ with positive weights.  Indeed, using Equation (\ref{eq:typeDact}),  we let $x$ be the element
$\mbox{diag}[x_{1},\dots, x_{l-1}, 0,0, -x_{l-1},\dots, -x_{1}]$ with $x_{j}\in \mathbb{Z}$ for $j=1,\dots, l-1$, and satisfying 
$x_{j}>0$ if $i_{j}=U$, $x_{j}<0$ if $i_{j}=L$, and $x_{j}\neq \pm x_{i}$ for $j\neq i$, and note that $x$ has the desired properties.  For $k\in \mathbb{Z}^{+}$, let $\fg(k)$ be the $k$-eigenspace for the action of $\ad(x)$.  
 Since $x\in\fk$, the Lie subalgebra $\fb:=\fh\oplus \oplus_{k\in\mathbb{Z}^{+}} \fg(k)$ is $\theta$-stable.  Further, since $x$ is regular semisimple in $\fg$, 
 $\fb$ is a Borel subalgebra, with nilradical $\fn=\bigoplus_{k\in\mathbb{Z}^{+}} \fg(k)$ (Lemma 3.84 of \cite{CoM}). 
By construction $\mathcal{C}_{i_{1},\dots, i_{l-1}}\subset \fn^{-\theta}$, and $\fn\in\mathcal{N}^{\theta}_{std}$. 
By Lemmas \ref{lem:nullconeirredcomponets} and \ref{l:dim},  $\dim \mathcal{C}_{i_{1},\dots, i_{l-1}} =\dim \fn^{-\theta}=r_{n-1}$, so 
that $\mathcal{C}_{i_{1},\dots, i_{l-1}}=\fn^{-\theta}$. 

For the converse, let $\fn\in\mathcal{N}^{\theta}_{std}$ with $\fn=[\fb,\fb]$, $\fb\in\mathcal{B}^{\theta}_{std}$.  Then the Cartan subgroup
$H\cap K$ of the Borel subgroup $B\cap K$ acts on $\fn$ and on $\fn^{-\theta}$ with positive weights.  
Hence, $\fn^{-\theta}\subset \N_{\fg^{-\theta}}$.  Further, $\fn^{-\theta}$ is a closed, irreducible subvariety of 
$\N_{\fg^{-\theta}}$ of dimension $\dim \fn^{-\theta}=r_{n-1}$ by Lemma \ref{l:dim}.  It follows that $\fn^{-\theta}$ is an irreducible component of $\mathfrak{N}_{\fg^{-\theta}}$ and hence 
$\fn^{-\theta}$ is of the form $\mathcal{C}_{i_{1},\dots, i_{l-1}}$ by Lemma \ref{lem:nullconeirredcomponets}. 
\end{proof}

Given Proposition \ref{p:nullconeslice}, we define an equivalence relation on $\mathcal{N}_{std}^{\theta}$ by:
\begin{equation}\label{eq:equivrelation}
\fn\equiv\fn^{\prime} \Leftrightarrow \fn^{-\theta}=(\fn^{\prime})^{-\theta}.
\end{equation}
Let $\mathfrak{S}^{-\theta}:=\mathcal{N}_{std}^{\theta}/\equiv$ be the set of equivalence classes of $\mathcal{N}_{std}^{\theta}$ modulo the relation in (\ref{eq:equivrelation}).  
We shall denote the equivalence class of $\fn\in\mathcal{N}^{\theta}_{std}$ by $[\fn]\in\mathfrak{S}^{-\theta}$.

\begin{cor}\label{c:countfn}
The cardinality of the set of equivalence classes $\mathfrak{S}^{-\theta}$ is $2^{r_{n-1}}$.
\end{cor}
\begin{proof}
The result follows by combining Lemma \ref{lem:nullconeirredcomponets} with Proposition \ref{p:nullconeslice}.
\end{proof}

We now use the Luna slice theorem and Proposition \ref{p:nullconeslice}
to completely describe the fibres $\Phi_{n}^{-1}(\Phi_{n}(x))$ for $x\in (\fh\cap\fk)_{reg}$.
\begin{thm}\label{thm:genericfibres}  
Let $x\in (\fh\cap\fk)_{reg}$.  The irreducible component decomposition of the fibre 
$\partialfibre$ is 
\begin{equation}\label{eq:partialfibredecomp}
\partialfibre=\displaystyle\bigcup_{[\fn]\in\mathfrak{S}^{-\theta}} \Ad(K)\cdot(x+\fn^{-\theta}),
\end{equation}
Further, each irreducible component $\Ad(K)\cdot(x+\fn^{-\theta})$ is smooth, and there are exactly
$2^{r_{n-1}}$ irreducible components in $\partialfibre$.

\end{thm}
\begin{proof}
By Part (2) of Proposition \ref{prop:flat}, the partial KW map $\Phi_{n}$ is a GIT quotient for $K$-action on $\fg$.  
Therefore Equation (\ref{eq:Lunaiso}) implies that we have a $K$-equivariant isomorphism
\begin{equation}\label{eq:Kiso}
\Kbundleslice\stackrel{\cong}{\to}\partialfibre\mbox{ given by } (k, x+n)\mapsto \Ad(k)\cdot (x+n),
\end{equation}
where $k\in K$ and $n\in\N_{\fg^{-\theta}}$.  Proposition \ref{p:nullconeslice} implies that the irreducible component 
decomposition of the fibre bundle $\Kbundleslice$ is: 
\begin{equation}\label{eq:bundlecpts}
\Kbundleslice=\displaystyle\bigcup_{[\fn]\in\mathfrak{S}^{-\theta}} K\times_{H\cap K} (x+\fn^{-\theta}).  
\end{equation}
Equation (\ref{eq:partialfibredecomp}) now follows from (\ref{eq:Kiso}) and (\ref{eq:bundlecpts}).  
The smoothness of the varieties $\Ad(K)\cdot(x+\fn^{-\theta})$ with $\fn\in\mathcal{N}^{\theta}_{std}$ follows from 
(\ref{eq:Kiso}) and the fact that the fibre bundles $K\times_{H\cap K} (x+\fn^{-\theta})$ are smooth.  The final
statement of the theorem follows from (\ref{eq:Kiso}), (\ref{eq:bundlecpts}), and Corollary \ref{c:countfn}.

\end{proof}

\begin{lem}\label{l:cpts}
Let $x\in(\fh\cap\fk)_{reg}$ and $\fn\in \mathcal{N}_{std}^{\theta}$.  Then 
\begin{equation}\label{eq:cpts}
\Ad(K)\cdot(x+\fn^{-\theta})=\Ad(K)\cdot(x+\fn).
\end{equation}
\end{lem}
\begin{proof}
Let $\fb\in\B^{\theta}_{std}$ with $\fn=[\fb,\fb]$.  Since $\fb\in\B^{\theta}_{std}$, $\fb\cap\fk$ is a Borel subalgebra of $\fk$ with nilradical $\fn\cap\fk$ and 
corresponding Borel subgroup $B\cap K$ of $K$.  Since $x\in(\fh\cap\fk)_{reg}$, then
$\Ad(B\cap K)\cdot x = x + \fn\cap\fk$ by Lemma 3.1.44 of \cite{CG}.   Hence,
since $\fn^{-\theta}$ is $B\cap K$-stable,
$$
\Ad(K)\cdot (x+\fn)=\Ad(K)\cdot(x+\fn\cap\fk+\fn^{-\theta})=\Ad(K)\cdot(\Ad(B\cap K)\cdot x+\fn^{-\theta})=\Ad(K)\cdot(x+\fn^{-\theta}).
$$
\end{proof}
Combining Lemma \ref{l:cpts} with Equations (\ref{eq:partialfibredecomp}) and (\ref{eq:cpts}), we can decompose 
the fibre $\partialfibre$, $x\in(\fh\cap\fk)_{reg}$ into irreducible subvarieties as:
\begin{equation}\label{eq:genericdecomp}
\partialfibre=\displaystyle\bigcup_{[\fn]\in\mathfrak{S}^{-\theta}} \Ad(K)\cdot (x+\fn).  
\end{equation}

The following result on closed $K$-orbits in $\B = \B_{\fg}$ is well-known
(figure 4.3 of \cite{Collingwood}).   The reader can find a proof in
Section 2.7 of \cite{CEspherical}.

\begin{prop}\label{p:closedKorbs}
(1) Let $\fg=\fso(2l+1)$, let $\fb_+$ be the upper triangular matrices in $\fg$, let $\fb_{-}:=\Ad(\dot{s}_{\alpha_{l}})\cdot \fb_{+}$,  
where $\dot{s}_{\alpha_{l}}$ is a representative of the simple 
reflection $s_{\alpha_{l}}$, and let $Q_{\fb_{+}}$ and $Q_{\fb_{-}}$ be the $K$-orbits on $\B$ containing $\fb_{+}$ and $\fb_{-}$ respectively. 
 Then the flag variety $\B$ has two closed $K$-orbits which are  $Q_{\fb_{+}}$ and $Q_{\fb_{-}}$.\\
(2)  Let $\fg=\fso(2l)$, and let $\fb_+$ be the set of upper triangular matrices in $\fg$.
Then $Q_{\fb_{+}}$ is the only closed $K$-orbit.
\end{prop}

\begin{rem}\label{r:thetastableclosed}
It is well-known that $\fb\in\B^{\theta}$ if and only if the $K$-orbit $Q_{\fb}$ of $\fb$ in $\B$ is closed in $\B$ 
(see Proposition 4.12 of \cite{CEexp}).  %Note that by Proposition \ref{p:closedKorbs}, every 
%closed $K$-orbit $Q$ on $\B$ contains a $\fb\in\B^{\theta}_{std}$. 
\end{rem}

%For $\fb\in\B$, we let $K\cdot\fb\subset\fg$ be the $K$-saturation of $\fb$ in $\fg$, and 
%fsimilarly $K\cdot\fn\subset\fg$ denotes the $K$-saturation of $\fn=[\fb,\fb]$.  %We denote the $K$-orbit of $\fb\in\B$ by $Q_{\fb}$. 

%Let $Q_{\fb}\subset\B$ be the $K$-orbit through $\fb$, and similarly for $Q_{\fb^{\prime}}$.  
\begin{lem}\label{l:Ksaturation}
Let $\fn,\,\fn^{\prime}\in \mathcal{N}^{\theta}_{std}$, and let $\fb,\, \fb^{\prime}\in \B^{\theta}_{std}$, with $\fn=[\fb,\fb]$ and $\fn^{\prime}=[\fb^{\prime},\fb^{\prime}]$.    
\begin{enumerate}
\item Then $\Ad(K)\cdot\fn=\Ad(K)\cdot\fn^{\prime}$ if and only if $Q_{\fb}=Q_{\fb^{\prime}}$.

\item If $Q_{\fb}\neq Q_{\fb^{\prime}}$, then $[\fn]\neq [\fn^{\prime}]$ in $\mathfrak{S}^{-\theta}$.

\end{enumerate}

\end{lem}
\begin{proof}
The sufficiency of (1) is clear.  For the necessity, suppose that $\Ad(K)\cdot\fn=\Ad(K)\cdot\fn^{\prime}$, and let $e\in\fn$ be principal nilpotent.  Then $e=\Ad(k)\cdot e^{\prime}$ for some $k\in K$, and $e^{\prime}\in\fn^{\prime}$ principal nilpotent.  Thus, $e\in \fb\cap\Ad(k)\cdot \fb^{\prime}$. 
But since $e$ is principal nilpotent, it is contained in a unique Borel subalgebra (Proposition 3.2.14 of \cite{CG}), forcing $\fb=\Ad(k)\cdot \fb^{\prime}.$  
Thus, $Q_{\fb}=Q_{\fb^{\prime}}$.  

To prove (2), suppose that $[\fn]=[\fn^{\prime}]\in\mathfrak{S}^{-\theta}$.  Then by Lemma \ref{l:cpts}, we have 
$\Ad(K)\cdot(x+\fn)=\Ad(K)\cdot(x+\fn^{\prime})$ for any $x\in (\fh\cap\fk)_{reg}$.  In particular, $\Ad(K)\cdot(\lambda x+\fn)=\Ad(K)\cdot(\lambda x+\fn^{\prime})$ 
for any $\lambda\in\C^{\times}$.  Taking the limit as $\lambda\to 0$, we obtain $\Ad(K)\cdot \fn=\Ad(K)\cdot \fn^{\prime}$, and the assertion in (2) now follows from (1).
\end{proof}

The following theorem determines the nilfibre.

\begin{thm}\label{thm:partialnil}
Let $\Phi_{n}:\fg\to \C^{r_{n-1}}\times \C^{r_{n}}$ be the 
partial KW map (Equation (\ref{eq:partial})).  

%Let $K\cdot\fb\subset\fg$ be the $K$-saturation of $\fb$, and let $K\cdot \fn\subset \fg$ be the 
%$K$-saturation of the nilradical $\fn$ of $\fb$.   
%The irreducible component decomposition of the variety 
%$\Phi_{n}^{-1}(0)$ is given by: 
%\begin{equation}\label{eq:partialnil}
%\Phi_{n}^{-1}(0)=\displaystyle\bigcup_{Q\mbox{ closed}} K\cdot \fn,
%\end{equation} 
%where the union is taken over all closed $K$-orbits $Q$ in $\B$. In particular:

\noindent If $\fg=\fso(2l)$,  then $\Phi_{n}^{-1}(0)=\Ad(K)\cdot\fn_{+}$ is
irreducible, where $\fn_{+}=[\fb_{+},\fb_{+}]$.

\noindent If $\fg=\fso(2l+1)$,  then $\Phi_{n}^{-1}(0)=\Ad(K)\cdot\fn_{+} \cup \Ad(K)\cdot \fn_{-}$ has two irreducible components,
where $\fn_{\pm}=[\fb_{\pm},\fb_{\pm}]$.
\end{thm}

\begin{proof}
We first prove that $\Phi_{n}^{-1}(0)=\cup_{[\fn]\in\mathfrak{S}^{-\theta}} \Ad(K)\cdot \fn$.
Let $x\in(\fh\cap\fk)_{reg}$, and $\lambda\in\C^{\times}$.  
Then by Equation (\ref{eq:genericdecomp}), 
\begin{equation}\label{eq:predeform}
\Phi_{n}^{-1}(\Phi_{n}(\lambda x))=\displaystyle\bigcup_{[\fn]\in\mathfrak{S}^{-\theta}} \Ad(K)\cdot (\lambda x+\fn).
\end{equation}
Now consider $\displaystyle\lim_{\lambda\to 0} \Phi_{n}^{-1}(\Phi_{n}(\lambda x))$.  
Since $\Phi_{n}$ is flat by Part (3) of Proposition \ref{prop:flat}, it follows from 
Theorem VIII.4.1 of \cite{SGA1} that 
$\displaystyle\lim_{\lambda\to 0} \Phi_{n}^{-1}(\Phi_{n}(\lambda x))=\Phi_{n}^{-1}\left (\displaystyle\lim_{\lambda\to 0} \Phi_{n}(\lambda x)\right ).$  
Therefore,
\begin{equation}\label{eq:limit1}
\displaystyle\lim_{\lambda\to 0} \Phi_{n}^{-1}(\Phi_{n}(\lambda x)) =\Phi_{n}^{-1}\left (\displaystyle\lim_{\lambda\to 0} \Phi_{n}(\lambda x)\right )= \Phi_{n}^{-1}(\Phi_{n}(0))=\Phi_{n}^{-1}(0).
\end{equation} 
On the other hand, 
\begin{equation}\label{eq:limit2}
\displaystyle\lim_{\lambda\to 0} \displaystyle\bigcup_{[\fn]\in\mathfrak{S}^{-\theta}} \Ad(K)\cdot (\lambda x+\fn)= \displaystyle \bigcup_{[\fn]\in\mathfrak{S}^{-\theta}} \Ad(K)\cdot (\displaystyle\lim_{\lambda\to 0} \lambda x+\fn)= \bigcup_{[\fn]\in\mathfrak{S}^{-\theta}} \Ad(K)\cdot \fn. 
\end{equation}
 By (\ref{eq:limit1}) and (\ref{eq:limit2}), we obtain
\begin{equation}\label{eq:gotit}
\Phi_{n}^{-1}(0)=\bigcup_{[\fn]\in\mathfrak{S}^{-\theta}} \Ad(K)\cdot \fn,
\end{equation}
%where the union is over nilradicals of Borel subalgebras in closed $K$-orbits in $\B$.
%YOU MAY NEED TO SAY SOMETHING HERE ABOUT HOW A THETA-STABLE NILRADICAL COMES FROM THETA STABLE BOREL..
Since each $\fn\in\mathcal{N}_{std}^{\theta}$ is stabilized by a Borel subgroup
$B\cap K$ of $K$, the set $\Ad(K)\cdot \fn$ is closed
by Lemma 39.2.1 of \cite{TY}.  Since 
each $\Ad(K)\cdot\fn$ is evidently irreducible, the distinct $\Ad(K)\cdot\fn$ form the irreducible components of
$\Phi_{n}^{-1}(0)$.  The theorem now follows from Remark \ref{r:thetastableclosed}, Lemma \ref{l:Ksaturation}, and the classification of the closed $K$-orbits on $\B$ given in Proposition \ref{p:closedKorbs}.  
\end{proof}

\begin{rem}\label{r:glncase}
For $\fg=\fgl(n)$ an analogous description of the nilfibre is proven in Proposition 3.10 of \cite{CEeigen}.  The reasoning used here gives a simpler
determination of the nilfibre than in \cite{CEeigen}.  
\end{rem}

\begin{rem}
In upcoming work, we use the Luna slice theorem and Theorem \ref{thm:partialnil} to describe arbitrary fibres of the 
partial KW map in terms of closed $K$-orbits on $\B$ and closed $K$-orbits on certain partial flag varieties.  We determine the $x\in \fg$ such that $\Phi^{-1}(\Phi(x))_{sreg}\neq \emptyset$ and develop geometric descriptions of $\Phi^{-1}(\Phi(x))_{sreg}$ 
using the theory of $K$-orbits on $\B$.  This is in the same spirit
as the main results of \cite{CEKorbs}.
\end{rem}

\subsection{The orthogonal KW nilfibre} 
We now use Theorem \ref{thm:partialnil} to show that $\Phi^{-1}(0)_{sreg}$
is empty.

\begin{prop}\label{p:nosreg}
Let $\fg=\fso(n)$, with $n>3$, and let $\Phi: \fg\to \C^{r_{2}}\times \dots\times \C^{r_{n}}$ be the KW map. 
Then $\Phi^{-1}(0)_{sreg}=\emptyset.$ 
\end{prop}
To show Proposition \ref{p:nosreg}, we first observe that the nilfibre of the partial KW map $\Phi_{n}^{-1}(0)$ has 
no $n$-strongly regular elements.  The key observation is the following 
proposition,  which can be viewed as an extension of Proposition 3.8 in \cite{CEKorbs}. 

\begin{prop}\label{prop:overlaps}  
Let $n>3$,
let $\fg=\fso(n)$, and let $K=SO(n-1)$.  Let $Q\subset \B$ be a closed $K$-orbit, and let $\fb\in Q$, with nilradical $\fn$.  
Then
%, and $\fz_{\fg}(\fn)$ be the centralizer of 
%$\fn$ is $\fg$ and similarly let $\fz_{\fk}(\fn\cap\fk)$ be the centralizer of %$\fn\cap\fk$ in $\fk$.
\begin{equation}\label{eq:fncentralizers}
\fz_{\fk}(\fn\cap\fk)\cap\fz_{\fg}(\fn)\neq 0.
\end{equation}

\end{prop}
\begin{proof}
By $K$-equivariance, it suffices to show (\ref{eq:fncentralizers}) for a representative $\fb$ of the closed $K$-orbit $Q$. By Proposition \ref{p:closedKorbs}, we can assume that the standard diagonal Cartan subalgebra $\fh$ is in $\fb$.  Let $\phi \in \Phi^{+}(\fg, \fh)$ be the highest root of $\fb$.  We claim for $n>4$ that $\phi$ is compact imaginary.  It then follows that the root space
$$
\fg_{\phi}\subset \fz_{\fk}(\fn\cap\fk)\cap\fz_{\fg}(\fn),
$$  
and the result follows.
For the claim, suppose first that $\fg=\fso(2l)$. By Part (2) of Proposition \ref{p:closedKorbs}, we can 
assume that $\fb=\fb_{+}$.  The highest root is then $\epsilon_{1}+\epsilon_{2}$, which is compact imaginary 
for $l>2$ (Example \ref{ex:roottypes}).  If $\fg=\fso(2l+1)$, then by Part (1) of Proposition \ref{p:closedKorbs}, we can assume 
that $\fb=\fb_{+}$ or $\fb=\fb_{-}=\Ad(\dot{s}_{\alpha_{l}})\cdot \fb_{+}$.  For both $\fb_{+}$
and $\fb_{-}$, the highest root is $\epsilon_{1}+\epsilon_{2}$, 
which is compact imaginary (Example \ref{ex:roottypes}).  

If $\fg=\fso(4)$, then $\phi=\epsilon_{1}+\epsilon_{2}$ is 
complex $\theta$-stable.  Since $\fn$ is abelian in this case, and $\fn$ is $\theta$-stable by Remark \ref{r:thetastableclosed},
 $(\fg_{\phi} \oplus \fg_{\theta(\phi)})^{\theta}\subset \fz_{\fk}(\fn\cap\fk)\cap \fz_{\fg}(\fn)$.

\end{proof}

\begin{cor}\label{c:nonsreg}
Let $n>3$, and let $\Phi_{n}:\fg\to \C^{r_{n-1}}\oplus \C^{r_{n}}$ be the partial Kostant-Wallach map.    
Then $\Phi_{n}^{-1}(0)$ contains no $n$-strongly regular elements. 
\end{cor}
\begin{proof}
Suppose $x \in \Phi_{n}^{-1}(0)$, so by Theorem \ref{thm:partialnil},
$x$ is contained in $\fn$, the nilradical of a Borel subalgebra
$\fb$ with $Q_{\fb}\subset\B$ closed.  
By Proposition \ref{prop:overlaps}, there is
a nonzero element $y$ of $\fz_{\fk}(\fn\cap \fk) \cap \fz_{\fg}(\fn)$.  By Remark \ref{r:thetastableclosed}, $x_{\fk}\in\fn\cap\fk$, so that $y\in\fz_{\fk}(x_{\fk})\cap \fz_{\fg}(x) $, and therefore $x$ is
not $n$-strongly regular.
\end{proof}

\begin{rem}\label{r:lowdim}
The assertion of the corollary is false
 for $n=3$.  In this case,  $\fso(3)\cong \fsl(2)$ and $\fk=\fh\subset \fsl(2)$, 
where $\fh$ is the standard Cartan subalgebra of $\fsl(2)$.  Further, the KW map for $\fso(3)\cong\fsl(2)$ coincides with 
the partial KW map for $\fso(3)\cong \fsl(2)$.  
In this case, it follows by Proposition 3.11 from \cite{CEeigen}
that each irreducible component of $\Phi_{n}^{-1}(0)$ contains strongly regular elements.
\end{rem} 

\begin{proof}[Proof of Proposition \ref{p:nosreg}]
It follows from Proposition \ref{prop:fullsreg} that $x\in\fg_{sreg}$ if and only if 
$x_{i}\in(\fg_{i})_{i-sreg}$ for all $i$.  Thus, if $x\in \Phi^{-1}(0)_{sreg}$, then $x_{i}\in \Phi_{i}^{-1}(0)_{i-sreg}$, for all $i=3,\dots, n$
where $\Phi_{i}:\fg_{i}\to\C^{r_{i-1}}\times \C^{r_{i}}$ is the partial KW map for $\fg_{i}$.  But Corollary \ref{c:nonsreg} 
implies that $\Phi_{i}^{-1}(0)_{i-sreg}=\emptyset$ for all $i=4,\dots, n$, and therefore $\Phi^{-1}(0)_{sreg}=\emptyset$. 
\end{proof} 

Let $I_{GZ}$ be the ideal of $\C[\fg]$ generated by the GZ functions 
$J_{GZ}$ in Equation (\ref{eq:GZfuns}).
We can use Proposition \ref{p:nosreg} to determine when the  
ideal $I_{GZ}$ is radical.  

\begin{cor}\label{c:notrad}
Let $\fg=\fso(n)$.  Then the ideal $I_{GZ}$ is radical if and only if $n=3$.  
\end{cor}
\begin{proof}
By Theorem 18.15(a) of \cite{Eis}, the ideal $I_{GZ}$ is radical
if and only if the set of differentials $\{df(x): f \in J_{GZ} \}$ is linearly independent on an open, dense subset 
of each irreducible component of $\Phi^{-1}(0)$.  It follows from Definition-Notation \ref{dfnote:sreg} that $I_{GZ}$ is radical 
if and only if each irreducible component of $\Phi^{-1}(0)$ contains strongly regular elements.  
But it follows from Proposition \ref{p:nosreg} and the case of $\fso(3)$ in Remark \ref{r:lowdim}
 that each irreducible component of $\Phi^{-1}(0)$ contains strongly regular elements if and only if $n=3$.     
\end{proof}

Using Proposition \ref{p:nosreg}, we can see that there is no orthogonal
analogue of the Hessenberg matrices, which play an important role for
$\fgl(n)$  (Equation (\ref{eq:Hessenberg})).

\begin{cor}\label{c:noHess}
Let $\fg=\fso(n)$ with $n>3$.  There is no subvariety $\mathfrak{X}\subset\fg$ such that 
the restriction of the KW map to $\mathfrak{X}$ is an isomorphism: 
$$
\Phi:\mathfrak{X}\to \C^{r_{2}}\times\dots \times \C^{r_{n}}.
$$
\end{cor}
\begin{proof}
The existence of such a subvariety $\mathfrak{X}$ would imply 
that every fibre of the KW map contained strongly regular elements, contradicting 
Proposition \ref{p:nosreg}.
\end{proof}

\begin{rem}
In upcoming work, we show that the KW map $\Phi$ is flat in both the general linear and 
orthogonal cases, so that $\Phi^{-1}(0)$ is an equidimensional variety of dimension 
$$
\dim \Phi^{-1}(0)=\dim\fg-\displaystyle\sum_{i=2}^{n} r_{i} =\frac{\dim \fg-r_{n}}{2}=\dim\fn
$$
using Equation (\ref{eq:sumri}).  This was previously known only for $\fgl(n)$
by work of Ovsienko \cite{Ov, Futfilt}.
We use this result to study the category of Gelfand-Zeitlin modules for the enveloping algebra $U(\fg)$ of $\fg$ studied by Futorny, Ovsienko, and others  \cite{DFO, Futfibres}.
This enables us to extend known results for Gelfand-Zeitlin modules
for $\fgl(n)$ to the orthogonal case.
\end{rem}

\bibliographystyle{amsalpha.bst}

\bibliography{bibliography-1}

\end{document}